\documentclass[11pt,a4paper]{amsart}

\usepackage[utf8]{inputenc}

\usepackage{fullpage}
\usepackage[british]{babel}
\usepackage{amsmath, amsthm, amssymb, mathtools}
\usepackage{enumerate}
\usepackage{color}
\usepackage{hyperref}

\newtheorem{theoremcounter}{Theorem Counter}[section]

\theoremstyle{definition}

\newtheorem{remark}[theoremcounter]{Remark}
\newtheorem{example}[theoremcounter]{Example}

\theoremstyle{plain}
\newtheorem{lemma}[theoremcounter]{Lemma}
\newtheorem{proposition}[theoremcounter]{Proposition}
\newtheorem{corollary}[theoremcounter]{Corollary}
\newtheorem{theorem}[theoremcounter]{Theorem}

\numberwithin{equation}{section}

\newcommand{\N}{\mathbb{N}}
\newcommand{\Z}{\mathbb{Z}}
\newcommand{\Q}{\mathbb{Q}}
\newcommand{\R}{\mathbb{R}}
\newcommand{\C}{\mathbb{C}}
\renewcommand{\H}{\mathbb{H}}
\newcommand{\bbP}{\mathbb{P}}

\newcommand{\Gr}{\mathfrak{D}}

\newcommand{\e}{\mathfrak{e}}

\newcommand{\calF}{\mathcal{F}}
\newcommand{\calG}{\mathcal{G}}
\newcommand{\calK}{\mathcal{K}}

\newcommand{\calQ}{\mathcal{Q}}

\renewcommand{\epsilon}{\varepsilon}
\newcommand{\eps}{\epsilon}
\newcommand{\vphi}{\varphi}

\DeclareMathOperator{\ImNew}{Im}
\renewcommand{\Im}{\ImNew}
\DeclareMathOperator{\ReNew}{Re}
\renewcommand{\Re}{\ReNew}
\DeclareMathOperator{\modNew}{mod}
\renewcommand{\mod}{\modNew}

\DeclareMathOperator{\SL}{SL}

\DeclareMathOperator{\Mp}{Mp}

\DeclareMathOperator{\sign}{sign}

\DeclareMathOperator{\tr}{tr}
\DeclareMathOperator{\reg}{reg}
\DeclareMathOperator{\ord}{ord}
\DeclareMathOperator{\CT}{CT}

\DeclareMathOperator{\Iso}{Iso}
\DeclareMathOperator{\Mat}{Mat}
\DeclareMathOperator{\FT}{FT}

\DeclareMathOperator{\hyperbolic}{hyp}
\DeclareMathOperator{\elliptic}{ell}
\DeclareMathOperator{\parabolic}{par}
\newcommand{\dhyp}{d_{\hyperbolic}}
\newcommand{\Eell}{E^{\elliptic}}
\newcommand{\Ehyp}{E^{\hyperbolic}}
\newcommand{\Epar}{E^{\parabolic}}

\newcommand{\pmat}[1]{\begin{pmatrix}#1\end{pmatrix}}
\newcommand{\smat}[1]{\bigl(\begin{smallmatrix}#1\end{smallmatrix}\bigr)}
\newcommand{\sprod}[1]{\left\langle#1\right\rangle}
\newcommand{\ssprod}[1]{\langle#1\rangle}

\newcommand{\abs}[1]{\left\lvert#1\right\rvert}

\newcommand{\snorm}[1]{\lVert#1\rVert}

\newcommand{\HypF}[2]{{}_{#1}F_{#2}}
\newcommand{\weightkaction}[1]{\,\Big\vert_{#1}\,}
\newcommand{\sweightkaction}[1]{\,\big\vert_{#1}\,}

\newcommand{\thmref}[2]{\hyperref[#2]{#1 \ref*{#2}}}
\renewcommand{\eqref}[1]{\hyperref[#1]{(\ref*{#1})}}

\renewcommand{\emph}[1]{\textbf{#1}}

\begin{document}

\title{Kronecker limit formulas for parabolic, hyperbolic and elliptic Eisenstein series via Borcherds products}
\author{Anna-Maria von Pippich}
\address{\rm Fachbereich Mathematik, Technische Universit\"at Darmstadt, Schlo{\upshape{\ss}}gartenstr. 7, 64289 Darmstadt, Germany}
\email{pippich@mathematik.tu-darmstadt.de}
\author{Markus Schwagenscheidt}
\address{\rm Fachbereich Mathematik, Technische Universit\"at Darmstadt, Schlo{\upshape{\ss}}gartenstr. 7, 64289 Darmstadt, Germany}
\email{schwagenscheidt@mathematik.tu-darmstadt.de}
\author{Fabian Völz}
\address{\rm Fachbereich Mathematik, Technische Universit\"at Darmstadt, Schlo{\upshape{\ss}}gartenstr. 7, 64289 Darmstadt, Germany }
\email{voelz@mathematik.tu-darmstadt.de}
\date{}

\maketitle

\begin{abstract}
	The classical Kronecker limit formula describes the constant term in the Laurent expansion at the first order pole of the non-holomorphic Eisenstein series associated to the cusp at infinity of the modular group. Recently, the meromorphic continuation and Kronecker limit type formulas were investigated for non-holomorphic Eisenstein series associated to hyperbolic and elliptic elements of a Fuchsian group of the first kind by Jorgenson, Kramer and the first named author. In the present work, we realize averaged versions of all three types of Eisenstein series for $\Gamma_0(N)$ as regularized theta lifts of a single type of Poincar\'e series, due to Selberg. Using this realization and properties of the Poincar\'e series we derive the meromorphic continuation and Kronecker limit formulas for the above Eisenstein series. The corresponding Kronecker limit functions are then given by the logarithm of the absolute value of the Borcherds product associated to a special value of the underlying Poincar\'e series.
\end{abstract}


\section{Introduction}

Let $N$ be a positive integer. The classical non-holomorphic Eisenstein series associated to a cusp $p$ of $\Gamma_{0}(N)$ is defined by
	\begin{align*}
		\Epar_p(z,s) = \sum_{M \in \Gamma_0(N)_p \backslash \Gamma_0(N)} \Im(\sigma_p^{-1}Mz)^s
	\end{align*}
	for $z \in \H$ and $s \in \C$ with $\Re(s)>1$. Here $\sigma_p \in \SL_2(\R)$ is a scaling matrix for the cusp $p$, that is $\sigma_p \smat{1&1\\0&1} \sigma_p^{-1}$ generates the stabilizer $\Gamma_0(N)_p / \{\pm1\}$ of the cusp $p$ in $\Gamma_{0}(N) / \{\pm1\}$. The Eisenstein series has a meromorphic continuation to the whole complex plane with a simple pole at $s = 1$, and the Kronecker limit formula describes the constant term in the Laurent expansion at this point. For example, for $N = 1$ it has the well-known form
\begin{align}\label{eq parabolic KLF}
\Epar_\infty(z,s) = \frac{3/\pi}{s-1} - \frac{1}{2\pi}\log\left(|\Delta(z)|\Im(z)^{6}\right) + C + O(s-1) \qquad \text{as $s \to 1$} ,
\end{align}
where $C = (6-72\zeta'(-1)-6\log(4\pi))/\pi$.
	We call $\Epar_p(z,s)$ a parabolic Eisenstein series in order to distinguish it from the following two analogs: 
	
	Given a geodesic $c$ in $\H$, i.e., a vertical line or a semi-circle centered at the real line, we define the hyperbolic Eisenstein series associated to $c$ via
	\begin{align*}
		\Ehyp_c(z,s) = \sum_{M \in \Gamma_0(N)_c \backslash \Gamma_0(N)} \cosh(\dhyp(Mz,c))^{-s}
	\end{align*}
	for $z \in \H$ and $s \in \C$ with $\Re(s)>1$. Here $\Gamma_0(N)_c$ denotes the stabilizer of the oriented geodesic $c$ in $\Gamma_0(N)$, and $\dhyp(z,c)$ denotes the hyperbolic distance from the point $z$ to $c$. They are scalar valued analogs of the form valued hyperbolic Eisenstein series introduced by Kudla and Millson in their work \cite{KudlaMillson1979}. The analytic continuation of $\Ehyp_c(z,s)$ for geodesics corresponding to hyperbolic elements of an arbitrary Fuchsian group of the first kind has been established by Jorgenson, Kramer and the first named author in \cite{JorgensonKramerPippich2009}. We note that we do not assume that the geodesic $c$ comes from some hyperbolic element of the underlying group $\Gamma_0(N)$, that is, the stabilizer $\Gamma_0(N)_c$ is allowed to be trivial, which is equivalent to saying that the image of $c$ in the modular curve $\Gamma_0(N) \backslash \H$ is an infinite geodesic. 
	
	For a point $w \in \H$, we define the elliptic Eisenstein series associated to $w$ via
	\begin{align*}
		\Eell_w(z,s) = \sum_{M \in \Gamma_0(N)_w \backslash \Gamma_0(N)} \sinh(\dhyp(Mz,w))^{-s}
	\end{align*}
	for $z \in \H$ not lying in the $\Gamma_{0}(N)$-orbit of $w$, and $s \in \C$ with $\Re(s)>1$. These series were introduced by Jorgenson and Kramer in their unpublished work \cite{KramerJorgenson2004} (see also \cite{KramerJorgenson2011}), and they have been investigated in detail for an arbitrary Fuchsian group of the first kind
	by the first named author in \cite{Pippich2010} and \cite{Pippich2016}. It was shown in \cite{Pippich2016} that the elliptic Eisenstein series has a meromorphic continuation to $\C$ and admits a Kronecker limit formula at $s = 0$, which for $N = 1$ takes the beautiful form
\[
\Eell_w(z,s) = -\log \left( \abs{j(z)-j(w)}^{2/|\Gamma_0(N)_w|} \right)\cdot s + O(s^{2}).
\]
Here $j(z)$ is the usual modular $j$-function. Further explicit examples are computed in \cite{JorgensonPippichSmajlovic2015}.

	The main goal of the present work is to realize (averaged versions of) all three types of Eisenstein series for 
	$\Gamma_0(N)$ as regularized theta lifts of certain non-holomorphic vector valued Poincar\'e series of weight $1/2$, and to use this representation to systematically derive explicit Kronecker limit type formulas for these averaged hyperbolic and elliptic Eisenstein series. For completeness, we also re-prove the classical Kronecker limit formula for the parabolic Eisenstein series using our theta lift approach. Let us describe our results in some more detail.
	
	Throughout this work, we assume that $N$ is squarefree. Although the results of this work certainly hold for general positive integers $N$ (with some minor modifications), this assumption greatly simplifies the exposition and allows us to make our results very explicit. 

	For $\beta \in \Z/2N\Z$ and a discriminant $D \in \Z$ with $D \equiv \beta^{2} \mod 4N$ we consider the non-holomorphic vector valued Poincar\'e series
\begin{align*}
	P_{\beta,D}(\tau,s)
	= \frac{1}{2} \sum_{(M,\phi) \in \ssprod{(T,1)} \backslash \Mp_2(\Z)} v^s e(D\tau/4N) \e_\beta \weightkaction{1/2,\rho} (M,\phi), \quad (\tau = u +iv \in \H),
\end{align*}
where we refer to Section \ref{section preliminaries} for the notation. Note that this definition slightly differs from the one given in \eqref{definition Pmb}, namely by the relation $D=4Nm$. The function $P_{\beta,D}(\tau,s)$ is a vector valued version of Selberg's Poincar\'e series introduced in \cite{Selberg1965}. It transforms like a vector valued modular form of weight $1/2$ for the Weil representation $\rho$ associated to a certain even lattice $L$ of signature $(2,1)$ and level $4N$, and it satisfies the differential equation
\begin{align*}
	\Delta_{1/2} P_{\beta,D}(\tau,s) = s \left(\frac{1}{2}-s\right) P_{\beta,D}(\tau,s) + s\frac{\pi D}{N} \, P_{\beta,D}(\tau,s+1),
\end{align*}
where $\Delta_{1/2}$ denotes the weight $1/2$ hyperbolic Laplace operator, see \eqref{LaplaceOperator}. For $D = 0$, the Poincar\'e series is a non-holomorphic parabolic Eisenstein series of weight $1/2$, whose analytic properties are well understood by the fundamental work of Selberg and Roelcke (see \cite{SelbergHarmonicAnalysis} and \cite{Roelcke1966,Roelcke1967}). In particular, it has a meromorphic continuation in $s$ to all of $\C$. For $D>0$, using the spectral theory of automorphic forms, Selberg proved in \cite{Selberg1965} that $P_{\beta,D}(\tau,s)$ has a meromorphic continuation to $\C$, with poles corresponding to the eigenvalues of $\Delta_{1/2}$. By computing the Fourier expansion of $P_{\beta,D}(\tau,s)$ and employing the estimates of Goldfeld and Sarnak (see \cite{GoldfeldSarnak1983}) and Pribitkin (see \cite{Pribitkin2000}) for the Kloosterman zeta functions appearing in this expansion, one obtains the meromorphic continuation of $P_{\beta,D}(\tau,s)$ also for $D < 0$. Carefully translating these classical results to our vector valued setting in Section \ref{section Poincare series}, we can evaluate the Fourier expansion of the Poincar\'e series at $s=0$, proving the following result (compare Theorem \ref{theorem MC of P via FE} and \ref{theorem P_beta,m(tau)}).

\begin{theorem} For each $\beta \in \Z/2N\Z$ and $D \in \Z$ with $D \equiv \beta^{2} \mod 4N$, the Poincar\'e series $P_{\beta,D}(\tau,s)$ has a meromorphic continuation in $s$ to $\C$ which is holomorphic at $s = 0$, yielding a harmonic Maass form $P_{\beta,D}(\tau,0)$ of weight $1/2$.
\end{theorem}

	Next, we consider Borcherds' regularized theta lift (see \cite{Borcherds}) of $P_{\beta,D}(\tau,s)$, namely
	\[
		\Phi(z,P_{\beta,D}(\,\cdot\,,s)) = \int_{\SL_{2}(\Z)\setminus \H}^{\reg}\langle P_{\beta,D}(\tau,s),\Theta(\tau,z) \rangle v^{1/2}   \frac{du \, dv}{v^{2}} , \quad (\tau = u + iv \in \H),
	\]
	where $\Theta(\tau,z)$ is the Siegel theta function associated to the lattice $L$, which transforms like a vector valued modular form of weight $1/2$ for the Weil representation $\rho$ in $\tau$, and is $\Gamma_{0}(N)$-invariant in $z$. Note that the above integral has to be regularized as explained in Section \ref{section borcherds products}.

	For $\beta \in \Z/2N\Z$ and $D \in \Z$ with $D \equiv \beta^{2} \mod 4N$ we let $\calQ_{\beta,D}$ be the set of integral binary quadratic forms $ax^{2}+bxy + cy^{2}$ of discriminant $D = b^{2}-4ac$ with $N \mid a$ and $b \equiv \beta \mod 2N$. The group $\Gamma_{0}(N)$ acts on $\calQ_{\beta,D}$ from the right by $Q.M = M^{t}QM$, with finitely many orbits if $D \neq 0$. Let $Q \in \calQ_{\beta,D}$ with $Q(x,y) = ax^{2}+bxy + cy^{2}$. If $D< 0$, then the order of the stabilizer of $Q$ in $\Gamma_0(N)$ is finite, and there is an associated Heegner (or CM) point
	\[
	z_{Q} = -\frac{b}{2a} + i \frac{\sqrt{|D|}}{2|a|} ,
	\]
	which is characterized by $Q(z_{Q},1) = 0$. If $D > 0$, then the stabilizer of $Q$ in $\Gamma_0(N)/\{\pm1\}$ is trivial if $D$ is a square, and infinite cyclic otherwise, and there is an associated geodesic in $\H$ given by
	\[
	c_{Q} = \{z \in \H: a|z|^{2} + bx + c = 0\} .
	\]
	In both cases the actions of $\Gamma_0(N)$ on $\calQ_{\beta,D}$ and $\H$ are compatible in the sense that $M.z_Q = z_{Q.M}$ and $M.c_Q = c_{Q.M}$ hold for all $M \in \Gamma_0(N)$. For $D < 0$ we let $H_{\beta,D}$ be the set of all Heegner points $z_{Q}$ with $Q \in \calQ_{\beta,D}$, and for $D \geq 0$ we let $H_{\beta,D} = \emptyset$. With the above notation we can now state one of our main results, which proven in {Section \ref{section lifts}}.
	
	\begin{theorem} \label{theorem lift introduction}
	For $s \in \C$ with $\Re(s) > 1$ the regularized theta lift $\Phi(z,P_{\beta,D}(\,\cdot\,,s))$ defines a real analytic function in $z \in \H \setminus H_{\beta,D}$ and a holomorphic function in $s$, which can be meromorphically continued to all $s \in \C$. It is holomorphic at $s =0$ if $D \neq 0$, and has a simple pole at $s = 0$ if $D = 0$. Further, we have
	\[
		\Phi(z,P_{\beta,D}(\,\cdot\,,s)) = \begin{dcases}
			\frac{2 \Gamma(s)}{(\pi D/N)^s} \sum_{Q \in \calQ_{\beta,D} / \Gamma_0(N)} \Ehyp_{c_Q}(z,2s) , & \text{if $D>0$} , \\
			4 N^s \zeta^*(2s) \sum_{p \in C(\Gamma_0(N))} \Epar_{p}(z,2s) , & \text{if $D=0$} , \\
			\frac{2 \Gamma(s)}{(\pi|D|/N)^s} \sum_{Q \in \calQ_{\beta,D} / \Gamma_0(N)} \Eell_{z_Q}(z,2s) , & \text{if $D<0$} ,
		\end{dcases}
	\]
	for $z \in \H \setminus H_{\beta,D}$ and $s \in \C$ with $\Re(s) > 1$. Here, $\zeta^*(s) = \pi^{-s/2} \Gamma(s/2) \zeta(s)$ is the completed Riemann zeta function.
\end{theorem}

%

It is remarkable that averaged versions of all three types of Eisenstein series arise as the theta lift of a single type of Poincar\'e series, being distinguished by the sign of $D$. We note that Matthes in \cite{Matthes1999}, Theorem 1.1, uses a similar Poincar\'e series to realize averaged versions of the hyperbolic kernel function $\sum_{M \in \Gamma_0(N)} \cosh(\dhyp(Mz,w))^{-s}$ as a theta lift. The relation between his and our result is explained by the identities given in \cite{JorgensonPippichSmajlovic}, Proposition 11 and 15.

The realization of individual hyperbolic and elliptic Eisenstein series for $\Gamma_0(N)$ as theta lifts will be presented in the upcoming thesis of the third named author (see \cite{Voelz}). These realizations also provide a conceptual approach
to all three types of Eisenstein series and will be used to define and study generalized hyperbolic and elliptic Eisenstein series on orthogonal groups.

Note that Theorem \ref{theorem lift introduction} also yields a new and unified proof of the meromorphic continuation of all three types of (averaged) Eisenstein series for the group $\Gamma_0(N)$. Since the continuations of the Eisenstein series are well documented in the literature, we do not to focus on this aspect here. Instead, we will employ Borcherds' theory of automorphic products developed in \cite{Borcherds} to establish explicit Kronecker limit type formulas for the averaged Eisenstein series given in Theorem \ref{theorem lift introduction}.

Using the functional equation of $\Epar_\infty(z,s)$ for $N = 1$, the classical Kronecker limit formula \eqref{eq parabolic KLF} is equivalent to the more attractive looking Laurent expansion
\[
\Epar_\infty(z,s) = 1 + \log\left(|\Delta(z)|^{1/6}\Im(z)\right)\cdot s + O(s^{2}) 
\]
at $s = 0$. Here $\Delta(z)$ is the unique normalized cusp form of weight $12$ for $\SL_2(\Z)$. In this article, we  establish Kronecker limit type formulas for the averaged Eisenstein series appearing in Theorem \ref{theorem lift introduction} at $s=0$. The corresponding Laurent expansions at $s=0$ are of the form
\[
a_0 + \calK(z) \cdot s + O(s^{2}) ,
\]
where $a_0 \in \C$ is a constant, and $\calK(z) \colon \H \to \C$ is some $\Gamma_{0}(N)$-invariant function. For brevity, we call $\calK(z)$ a Kronecker limit function.
It is well known that the hyperbolic and elliptic Eisenstein series vanish at $s =0$, yielding $a_0=0$ in these cases. The explicit computation of the Kronecker limit function $\calK(z)$ consists of the following three main steps:

\begin{enumerate}
	\item Firstly, we explicitly determine the functions $P_{\beta,D}(\tau,0)$. They turn out to be of rather different nature for different signs of $D$. For $D < 0$, the Poincar\'e series is in general a properly non-holomorphic harmonic Maass form which is determined by its principal part, for $D = 0$ it is a holomorphic modular form which can explicitly be written as a linear combination of unary theta functions, and for $D > 0$ it is a cusp form which is characterized by the fact that the Petersson inner product with a cusp form $f$ of weight $1/2$ for $\rho$ essentially gives the $(\beta,D)$-th Fourier coefficient of $f$. We refer to Theorem \ref{theorem P_beta,m(tau)} for the details.
	\item Next, we show that the functions 
	\[
	\Phi(z,P_{\beta,D}(\,\cdot\,,s))|_{s = 0} \quad \text{and} \quad \Phi(z,P_{\beta,D}(\,\cdot\,,0))
	\]
	essentially agree (some care is necessary for $D = 0$), see Proposition \ref{proposition theta lift continuation}. Since $\Gamma(s)$ and $\zeta^{*}(2s)$ have a simple pole at $s = 0$, Theorem \ref{theorem lift introduction} then implies that the Kronecker limit function $\calK(z)$ is basically given by the theta lift $\Phi(z,P_{\beta,D}(\,\cdot\,,0))$.
	\item By the theory of automorphic products developed in \cite{Borcherds} and \cite{BruinierOno}, it is known that the theta lift $\Phi(z,P_{\beta,D}(\,\cdot\,,0))$ is essentially given by the logarithm of the absolute value of the Borcherds product associated to $P_{\beta,D}(\tau,0)$. Our explicit description of the functions $P_{\beta,D}(\tau,0)$ enables us to determine the required Borcherds products, which in turn gives the Kronecker limit functions $\calK(z)$.
\end{enumerate}
	
	In the following, we present the three Kronecker limit type formulas we obtained via the above process, depending on the sign of $D$. However, since the Kronecker limit functions for the averaged elliptic and hyperbolic Eisenstein series look quite technical for arbitrary squarefree integers $N$, we only state simplified versions in these cases in the introduction, restricting to special values of $N$. For the general theorems we refer 
	to Section \ref{section Kronecker limit functions}.
	
	For $D=0$ the parabolic Kronecker limit function generalizes the classical Kronecker limit formula for $N=1$ seen above:
	
	\begin{theorem}
		At $s=0$ we have the Laurent expansion
		\begin{align*}
			\sum_{p \in C(\Gamma_{0}(N))} \Epar_p(z,s) 
			= 1+\frac{1}{\sigma_{0}(N)}\sum_{c \mid N}\log\left(\left| \Delta(cz)\right|^{1/6}\Im(z)\right)\cdot s + O(s^{2}) .
		\end{align*}
	\end{theorem}
	
	Here $C(\Gamma_{0}(N))$ denotes the set of cusps of $\Gamma_0(N)$. In fact, this is the Kronecker limit formula of the parabolic Eisenstein series for the generalized Fricke group $\Gamma_{0}^{*}(N)$, which is the extension of $\Gamma_{0}(N)$ by all Atkin-Lehner involutions, compare \cite{JorgensonSmajlovicThen}.
	
	In the hyperbolic case, we will see that the Kronecker limit function for the averaged hyperbolic Eisenstein series vanishes for trivial reasons if $N = 1$ or $N = p$ is a prime, or if $D$ is not a square. Thus it is reasonable to assume that $N$ is the product of at least two different primes, and that $D$ is a square, in order to obtain an interesting statement. In the following theorem we deal with the simplest non-trivial situation.
	
	\begin{theorem}
		Let $\beta=n$ and $D=n^2$ for some positive integer $n$, and let $N=pq$ be the product of two different primes. Then the averaged hyperbolic Eisenstein series admits a Laurent expansion at $s = 0$ of the form
		\[
			\sum_{Q \in \calQ_{n,n^{2}}/\Gamma_{0}(N)} \Ehyp_{c_{Q}}(z,s)
			= \calK(z) \cdot s + O(s^2) ,
		\]
		where the Kronecker limit function $\calK(z)$ is given by
		\[
			\calK(z) = \begin{dcases}
				\frac{6n}{\vphi(N)} \log\abs{ \frac{\eta(z) \eta(Nz)}{\eta(pz) \eta(qz)} } , & \text{if $(n,N)=1$} , \\
				0 , & \text{if $(n,N)>1$} .
			\end{dcases}
		\]
		Here $\varphi(N) = (p-1)(q-1)$ is Euler's totient function and $\eta(z) = e(z/24)\prod_{n \geq 1}(1-e(nz))$ is the Dedekind eta function.
	\end{theorem}
		
	For general squarefree $N$, the Kronecker limit function is given by the logarithm of the absolute value of an eta quotient of weight $0$, but the exponents of the quotient become much more complicated. We refer to Theorem \ref{theorem hyperbolic KLF} for the general statement.
	
	In order to present a Kronecker limit formula in the elliptic case, we recall that there are finitely many $N$ such that the generalized Fricke group $\Gamma_{0}^{*}(N)$ has genus $0$, for example, the set of all such primes is given by
	\[
		\{2,3,5,7,11,13,17,19,23,29,31,41,47,59,71\} .
	\] 
	For such $N$ let $j_{N}^{*} = e(-z) + O(e(z)) \in M_{0}^{!}(\Gamma_{0}^{*}(N))$ be the corresponding normalized Hauptmodul for $\Gamma_{0}^{*}(N)$. In this situation we obtain the following result.
	
	\begin{theorem} \label{theorem elliptic KLF introduction}
		Let $D<0$, and let $N$ be a squarefree positive integer such that the group $\Gamma_{0}^{*}(N)$ has genus $0$. Then the averaged elliptic Eisenstein series has a Laurent expansion at $s = 0$ of the form
		\[
			\sum_{Q \in \calQ_{\beta,D}/\Gamma_{0}(N)}\Eell_{z_{Q}}(z,s) = - \frac{1}{\sigma_{0}(N)}\sum_{Q \in \calQ_{\beta,D}/\Gamma_{0}(N)}\log\left(\big|j_{N}^{*}(z)-j_{N}^{*}(z_{Q})\big|^{2/|\Gamma_0(N)_Q|}\right)\cdot s + O(s^{2}),
		\]
		where $\sigma_{0}(N) = \sum_{d \mid N}1$ is the number of positive divisors of $N$.
	\end{theorem}
	
	For general squarefree $N$, the Kronecker limit function is given by the logarithm of the absolute value of a holomorphic function $\Psi_{\beta,D}$ on $\H$ which transforms like a modular form of weight $0$ for some unitary character, and which is determined by the orders of its roots at the Heegner points $z_{Q}$ for $Q \in Q_{\beta,D}$ and its orders at the cusps of $\Gamma_{0}(N)$. See Theorem \ref{theorem elliptic KLF} for the details.
	
	\medskip
	
	We now describe the organization of our work. In Section 2 we introduce the basic concepts used in this presentation, such as the Grassmannian model of the upper-half plane and its underlying lattice, vector valued harmonic Maass forms and unary theta functions, and Borcherds' theory of regularized theta lifts and automorphic products.
	
	Afterwards, we study a vector valued version of Selberg's Poincar\'e series in Section 3. We use spectral theory to establish the meromorphic continuation of the corresponding Kloosterman zeta functions involved in the Fourier expansion of the Poincar\'e series, which enables us to continue our Poincar\'e series to the whole complex plane. In particular, we can explicitly describe the Poincar\'e series at the particular point $s=0$.
	
	In Section 4 we prove the existence and continuation of the regularized theta lift of the above Poincar\'e series, showing that the evaluation of the continuation of the lift at $s=0$ essentially agrees with the lift of the evaluation of the Poincar\'e series at this point. Moreover, we determine the regularized theta lift of the Poincar\'e series via unfolding against the Poincar\'e series, seeing that it actually can be written as an averaged sum of hyperbolic, parabolic or elliptic Eisenstein series, depending on the parameter of the Poincar\'e series.
	
	Finally, in Section 5 we use this identity between the theta lift of our Poincar\'e series and the different non-holomorphic Eisenstein series to obtain Kronecker limit type formulas for all three types of Eisenstein series. As these are of rather different nature we study them separately, starting with the classical parabolic case. The general hyperbolic and elliptic Kronecker limit formulas are then given in Theorem \ref{theorem hyperbolic KLF} and Theorem \ref{theorem elliptic KLF}. We end both sections with some comments on these formulas.
	
	\medskip
	
	We thank Jan H. Bruinier for helpful discussions.

\section{Preliminaries}\label{section preliminaries}

\subsection{The modular curve $\Gamma_0(N) \backslash \H$}

The group $\SL_2(\R)$ acts on the complex upper-half plane $\H$ by fractional linear transformations. In the present work we focus on the congruence subgroup $\Gamma_0(N)$, consisting of integer matrices with determinant $1$ and lower left entry being divisible by $N$. In particular, the group $\Gamma_0(N)$ is a Fuchsian group of the first kind. In order to simplify the presentation we assume that $N$ is squarefree. In this case the cusps of $\Gamma_0(N)$ can be represented by the fractions $1/c$ where $c$ runs through the positive divisors of $N$, and the width of the cusp $1/c$ is given by $N/c$. We denote the set of cusps of $\Gamma_0(N)$ by $C(\Gamma_0(N))$. We also note that for $N$ being squarefree the index of $\SL_2(\Z)$ in $\Gamma_0(N)$ is simply given by $\sigma_1(N)$.

\subsection{The Grassmannian model of the upper half-plane}

We consider the rational quadratic space $V$ of signature $(2,1)$ given by the trace zero matrices in $\Mat_2(\Q)$ together with the quadratic form $Q(X)=-N\det(X)$. The associated bilinear form is $(X,Y)=N\tr(XY)$. The group $\SL_2(\Q)$ acts on $V$ by $g.X = g X g^{-1}$. We let $\Gr$ be the Grassmannian of $2$-dimensional subspaces in $V(\R) = V\otimes \R$ on which the quadratic form $Q$ is positive definite, and we identify $\Gr$ with the complex upper half-plane $\H$ by associating to $z = x + iy \in \H$ the orthogonal complement $X(z)^\perp$ of the negative line $\R X(z)$ generated by
 \[
	 X(z)=\frac{1}{\sqrt{N}y} \begin{pmatrix} -x &|z|^{2} \\  -1 & x \end{pmatrix}.
 \]
Since $g.X(z)=X(gz)$ for $g\in \SL_{2}(\R)$ and $z \in \H$ the above identification of $\Gr$ and $\H$ is $\SL_{2}(\R)$-equivariant.

In $V$ we consider the lattice
\[
	L = \left\{ \begin{pmatrix}b & c/N \\ -a & -b\end{pmatrix}: a,b,c \in \Z\right\}.
\]
Its dual lattice is
\[
	L' = \left\{ \begin{pmatrix}b/2N & c/N \\ -a & -b/2N\end{pmatrix}: a,b,c  \in\Z\right\}.
\]
We see that $L'/L$ is isomorphic to $\Z/2N\Z$, equipped with the quadratic form $x \mapsto x^{2}/4N \mod \Z$. Note that the group $\Gamma_{0}(N)$ acts on $L$ and fixes the classes of $L'/L$.
	
For fixed $\beta \in L'/L$ and $m \in \Z + Q(\beta)$ we let
\begin{align*}
	L_{\beta,m} = \{X \in L+\beta: Q(X) = m\}.
\end{align*}
If $m \neq 0$ the group $\Gamma_{0}(N)$ acts on $L_{\beta,m}$ with finitely many orbits, and elements $X \in L_{\beta,m}$ with $X = \smat{b/2N&c/N\\-a&-b/2N}$ correspond to integral binary quadratic forms
\[
	Q_X(x,y) = aNx^{2} +bxy + cy^{2}
\]
of discriminant $4Nm$ with $b \equiv \beta \mod 2N$. This identification is compatible with the corresponding actions of $\Gamma_{0}(N)$ in the sense that $Q_{X}.g = Q_{g^{-1}.X}$ for $g \in \Gamma_{0}(N)$. In particular, we have a bijection between $\Gamma_{0}(N) \backslash L_{\beta,m}$ and $\mathcal{Q}_{\beta,4Nm}/\Gamma_{0}(N)$. Moreover, if $m>0$ or $m<0$ we associate to $X \in L_{\beta,m}$ the Heegner geodesic $c_X = c_{Q_X}$ or the Heegner point $z_X = z_{Q_X}$, respectively. Clearly, this identification is again compatible with the corresponding actions of $\Gamma_{0}(N)$.
	
The set $\Iso(V)$ of isotropic lines in $V$ can be identified with $\bbP^{1}(\Q)$ via the map
\begin{align} \label{eq cusps and isotropic lines}
	\bbP^{1}(\Q) \to \Iso(V) , \quad (\alpha:\beta) \mapsto \pmat{-\alpha \beta & \alpha^2 \\ - \beta^2 & \alpha \beta} ,
\end{align}
and this identification respects the corresponding actions of $\Gamma_0(N)$. In particular, it gives rise to a bijection between $\Gamma_{0}(N)\setminus \Iso(V)$ and the set of cusps of $\Gamma_{0}(N)$. We write $p_X$ for the cusp associated to the isotropic line $\Q X$. Further, we choose a generator $X_\ell \in \ell$ with $\ell \cap L = \Z X_\ell$ for every line $\ell \in \Iso(V)$. Clearly, $X_\ell$ is unique up to a sign.

\subsection{Harmonic Maass forms}
	
We let $\Mp_{2}(\R)$ be the metaplectic group consisting of all pairs $(M,\phi)$ with $M = \left(\begin{smallmatrix}a & b \\ c & d \end{smallmatrix}\right) \in \SL_{2}(\R)$ and $\phi \colon \H \to \C$ holomorphic such that $\phi^{2}(\tau) = c\tau + d$. Further, we denote the preimage of $\SL_2(\Z)$ under the covering map $(M,\phi) \mapsto M$ by $\Mp_2(\Z)$. For $\beta \in L'/L$ we let $\e_{\beta}$ be the standard basis vectors of the group ring $\C[L'/L]$, and we let $\langle \cdot,\cdot \rangle$ be the inner product on $\C[L'/L]$ which satisfies $\langle \e_{\beta},\e_{\gamma}\rangle = \delta_{\beta,\gamma}$ and is antilinear in the second variable. The associated Weil representation $\rho$ is defined on the generators $(T,1)$ and $(S,\sqrt{\tau})$ of $\Mp_{2}(\Z)$ by
\[
	\rho(T,1) \e_{\beta}
		= e(Q(\beta)) \e_{\beta} , \qquad
	\rho(S,\sqrt{\tau}) \e_{\beta}
		= \frac{e(-1/8)}{\sqrt{|L'/L|}} \sum_{\gamma \in L'/L} e(-(\beta,\gamma)) \e_{\gamma} .
\]
Here $T=\smat{1&1\\0&1}$, $S=\smat{0&-1\\1&0}$ and $e(z) = e^{2\pi i z}$ for $z\in \C$, as usual. The dual Weil representation is denoted by $\bar{\rho}$.
		
Recall from \cite{BruinierFunke04} that a smooth function $f: \H \to \C[L'/L]$ is called a harmonic Maass form of weight $k \in 1/2+\Z$ for $\rho$ if it is annihilated by the weight $k$ hyperbolic Laplace operator $\Delta_k$ given by
\begin{align}\label{LaplaceOperator}
	\Delta_k = -v^{2}\left(\frac{\partial^{2}}{\partial u^{2}} + \frac{\partial^{2}}{\partial v^{2}}\right) + ikv \left( \frac{\partial}{\partial u} + i\frac{\partial}{\partial v}\right) , \quad (\tau = u + iv),
\end{align}
if it is invariant under the weight $k$ slash operator given by
\[
	f \sweightkaction{k,\rho}(M,\phi) = \phi(\tau)^{-2k}\rho^{-1}(M,\phi)f(M\tau)
\]
for all $(M,\phi) \in \Mp_{2}(\Z)$, and if it grows at most linearly exponentially at $\infty$. We denote the space of harmonic Maass forms by $H_{k,\rho}$, and we let $M^{!}_{k,\rho}$ be the subspace of weakly holomorphic modular forms, consisting of the forms in $H_{k,\rho}$ which are holomorphic on $\H$. 
		
The antilinear differential operator $\xi_k f=2iv^{k}\overline{\frac{\partial}{\partial \bar{\tau}}f(\tau)}$ defines a surjective map $\xi_k: H_{k,\rho}\rightarrow M^{!}_{2-k,\bar{\rho}}$ with kernel $M^{!}_{k,\rho}$. We let $H^{+}_{k,\rho}$ be the space of harmonic Maass forms which map to cusp forms under $\xi_{k}$. A form $f \in H_{k,\rho}^{+}$ has a Fourier expansion of the form
\begin{align}\label{FourierExpansion}
	f(\tau) = \sum_{\beta \in L'/L} \Bigg( \sum_{\substack{m \in \Z+Q(\beta) \\ m \gg -\infty}} c_{f}^{+}(\beta,m) e(m\tau) + \sum_{\substack{m \in \Z+Q(\beta) \\ m < 0}} c_{f}^{-}(\beta,m) \Gamma(1-k,4\pi |m|v) e(m\tau) \Bigg) \e_{\beta} ,
\end{align}
where $\Gamma(s,x) = \int_{x}^{\infty}t^{s-1}e^{-t}dt$ denotes the incomplete gamma function. The finite Fourier polynomial $\sum_{\beta \in L'/L}\sum_{m \leq 0}c_{f}^{+}(\beta,m)e(m\tau) \mathfrak{e}_{\beta}$ is called the principal part of $f$.

The regularized inner product of $f \in H^{+}_{k,\rho}$ and $g \in M_{k,\rho}$  is defined by
\begin{align}\label{PeterssonInnerProduct}
	(f,g)^{\reg} = \lim_{T \to \infty} \int_{\mathcal{F}_{T}}\langle f(\tau),g(\tau) \rangle v^{k}\frac{du\, dv}{v^{2}},
\end{align}
where $\mathcal{F}_{T} = \{\tau \in \H: |u|\leq \tfrac{1}{2}, |\tau| \geq 1 ,v \leq T\}$ is a truncated fundamental domain for the action of $\SL_{2}(\Z)$ on $\H$. If the integral converges without the regularization (e.g., if $f$ and $g$ are cusp forms), this is just the usual Petersson inner product $(f,g)$.

%

\subsection{Unary theta functions} \label{section unary theta functions}

We let $K$ be the one-dimensional positive definite sublattice of $L$ generated by the vector $\smat{1&0\\0&-1}$. Its dual lattice $K'$ is generated by $\frac{1}{2N}\smat{1&0\\0&-1}$. We see that $K'/K \cong L'/L$, so modular forms for the Weil representation of $K$ are the same as modular forms for $\rho$. Thus the unary theta function associated to the lattice $K$ given by
\[
	\theta(\tau) = \sum_{\beta \in K'/K} \sum_{X \in K +\beta} e(Q(X)\tau) \e_{\beta} = \sum_{\beta (2N)}\sum_{\substack{n \in \Z \\ n \equiv \beta (2N)}}e(n^{2}\tau/4N)\e_{\beta}
\]
is a modular form of weight $1/2$ for $\rho$ by Theorem 4.1 in \cite{Borcherds}.
	
The orthogonal group $O(L'/L)$ acts on vector valued functions $f = \sum_{\beta} f_{\beta} \e_{\beta}$ modular of weight $1/2$ for $\rho$ by $f^{w} = \sum_{\beta} f_{\beta} \e_{w(\beta)}$. The elements of $O(L'/L)$ are all involutions, so-called Atkin-Lehner involutions, and as $N$ is squarefree they correspond to the positive divisors $c$ of $N$. More precisely, the automorphism $w_{c}$ corresponding to $c \mid N$ is defined by the equations
\begin{align}\label{AtkinLehnerInvolution}
	w_{c}(\beta) \equiv -\beta \ (2c) \quad \text{and} \quad w_{c}(\beta) \equiv \beta \ (2N/c)
\end{align}
for $\beta \in L'/L \cong \Z/2N\Z$ (compare \cite{EichlerZagier}, Theorem 5.2). We also note that $f^{w_{c}} = f^{w_{N/c}}$ and $(f^{w_{c}},g) = (f,g^{w_{c}})$ for $f,g$ modular of weight $1/2$ for $\rho$.

Using a dimension formula for $M_{1/2,\rho}$ by Skoruppa, it is easy to show that the unary theta functions $\theta^{w_{c}}$, with $c$ running through the positive divisors of $N$ modulo the relation $c \sim N/c$, form a basis of $M_{1/2,\rho}$ (compare \cite{BruinierSchwagenscheidt}, Lemma 2.1).

\subsection{Borcherds products}\label{section borcherds products}

The Siegel theta function associated to the lattice $L$ is defined by
\[
	\Theta(\tau,z) = v^{1/2}\sum_{\beta \in L'/L} \ \sum_{X \in L+\beta} e(Q(X_{z})\tau + Q(X_{z^{\perp}})\bar{\tau}) \e_\beta,
\]
where $X_z$ denotes the orthogonal projection of $X$ onto the positive definite subspace $X(z)^\perp$, and $X_{z^\perp}$ the projection of $X$ onto the negative line $\R X(z)$. By \cite{Borcherds}, Theorem 4.1, the Siegel theta function is $\Gamma_{0}(N)$-invariant in $z$, and transforms like a modular form of weight $1/2$ for $\rho$ in $\tau$. 

Given a function $f \colon \H \to \C[L'/L]$ modular of weight $1/2$ with respect to the Weil representation $\rho$, Borcherds' regularized theta lift of $f$ is defined by
\begin{align} \label{eq def theta lift}
	\Phi(z,f) = \CT_{t=0} \left[ \lim_{T \to \infty} \int_{\calF_T} \sprod{f(\tau),\Theta(\tau,z)} v^{1/2-t} \frac{du \, dv}{v^{2}} \right],
\end{align}
where $\CT_{t=0}F(t)$ denotes the constant term of the Laurent expansion of the analytic continuation of $F(t)$ at $t=0$. The lift was studied by Borcherds in \cite{Borcherds} for weakly holomorphic forms $f \in M_{1/2,\rho}^{!}$, and generalized by Bruinier and Ono \cite{BruinierOno} to harmonic Maass forms $f \in H_{1/2,\rho}^{+}$. It turns out that for such $f$, the lift $\Phi(z,f)$ defines a $\Gamma_{0}(N)$-invariant real analytic function with logarithmic singularities at certain Heegner points, which are determined by the principal part of $f$.

The so-called Weyl vector $\rho_{f,1/c}$ associated to a harmonic Maass form $f \in H_{1/2,\rho}^{+}$ and the cusp $1/c$ with $c \mid N$ is defined by
\begin{align}\label{WeilVector}
	\rho_{f,1/c} = \frac{\sqrt{N}}{8\pi}(f,\theta^{w_{c}})^{\reg}.
\end{align}
Note that $\theta^{w_c}$ can also be seen as the unary theta function associated to the one-dimensional positive definite sublattice of $L$ corresponding to the cusp $1/c$ (compare \cite{BruinierOno}, Section 4.1).
	
\begin{theorem}[\cite{BruinierOno}, Theorem 6.1] \label{theorem Borcherds products}
	Let $f \in H_{1/2,\rho}^{+}$ be a harmonic Maass form with real coefficients $c_{f}^{+}(\beta,m)$ for all $\beta \in L'/L$ and $m \in \Z + Q(\beta)$. Moreover, assume that $c_{f}^{+}(\beta,m) \in \Z$ for $m \leq 0$. The infinite product
	\begin{align}\label{BorcherdsProduct}
		\Psi(z,f) = e(\rho_{f,\infty}z) \prod_{n=1}^\infty \big(1-e(nz)\big)^{c^{+}_{f}(n,n^{2}/4N)}
	\end{align}
	converges for $\Im(z)$ sufficiently large and has a meromorphic continuation to all of $\H$ with the following properties:
	\begin{enumerate}
		\item It is a meromorphic modular form of weight $c_f^+(0,0)$ for $\Gamma_{0}(N)$ with a unitary character which may have infinite order.
		\item The orders of $\Psi(z,f)$ in $\H$ are determined by the Heegner divisor
		\[
		\frac{1}{2}\sum_{\beta \in L'/L}\sum_{\substack{m \in \Z + Q(\gamma) \\ m < 0}}c_{f}^{+}(\beta,m)\sum_{X \in L_{\beta,m}}(z_{X}).
		\]
		\item The order of $\Psi(z,f)$ at the cusp $1/c$ for $c \mid N$ is given by
			\[
				\ord_{1/c}(\Psi(z,f)) = \frac{c}{N}\rho_{f,1/c}.
			\]
		\item The regularized theta lift of $f$ is given by
			\begin{align*}
				\Phi(z,f) = -c_f^{+}(0,0)(\log(4\pi N)+\Gamma'(1)) - 4\log \abs{ \Psi(z,f) \Im(z)^{c_f^{+}(0,0)/2} } .
			\end{align*}
	\end{enumerate}
\end{theorem}
\section{Selberg's Poincar\'e series} \label{section Poincare series}

In the following we define a vector valued version of the non-holomorphic Poincar\'e series introduced by Selberg in his famous work \cite{Selberg1965}. Following the ideas of Selberg we will use spectral theory and Fourier analysis to prove a meromorphic continuation of this Poincaré series. Finally, we will evaluate the continued series at the special value $s=0$.

Given $\beta \in L'/L$ and $m \in \Z+Q(\beta)$, we define the non-holomorphic Poincar\'e series of index $(\beta,m)$ as
\begin{align} \label{definition Pmb}
	P_{\beta,m}(\tau,s)
	= \frac{1}{2} \sum_{(M,\phi) \in \ssprod{(T,1)} \backslash \Mp_2(\Z)} v^s e(m\tau) \e_\beta \weightkaction{1/2,\rho} (M,\phi)
\end{align}
for $\tau = u+iv \in \H$ and $s \in \C$ with $\Re(s)>3/4$. The sum is absolutely convergent and defines an analytic function in $s$, which is by construction modular of weight $1/2$ with respect to $\rho$. However, $P_{\beta,m}(\tau,s)$ is not an eigenfunction of the hyperbolic Laplace operator, but it satisfies the differential equation
\begin{align} \label{eq DE of P}
	\Delta_{1/2} P_{\beta,m}(\tau,s) = s \left(\frac{1}{2}-s\right) P_{\beta,m}(\tau,s) + 4\pi ms \, P_{\beta,m}(\tau,s+1) .
\end{align}
In order to study the meromorphic continuation of $P_{\beta,m}(\tau,s)$ we also need to investigate the Kloosterman zeta functions appearing in its Fourier expansion. In fact, in \cite{Selberg1965}, \cite{GoldfeldSarnak1983} and \cite{Pribitkin2000} these Poincar\'e series were studied for the purpose of finding good growth estimates for the Kloosterman zeta functions. However, the case of $\Delta_k$ having continuous spectrum was mostly neglected.

Given $\beta,\gamma \in L'/L$ and $m \in \Z+Q(\beta)$, $n \in \Z+Q(\gamma)$, we define the Kloosterman zeta function by
\begin{align} \label{def Kloosterman zeta}
	Z(s;\beta,m,\gamma,n) = \sum_{c \neq 0} |c|^{1-2s} H_c(\beta,m,\gamma,n)
\end{align}
for $s \in \C$ with $\Re(s)>1$. Here $H_c(\beta,m,\gamma,n)$ is the generalized Kloosterman sum of weight $k=1/2$ defined for example in \cite[eq. (1.38)]{Bruinier2002}, that is
\begin{align} \label{def Kloosterman sum}
	H_c(\beta,m,\gamma,n) = \frac{e(-\sign(c)k/4)}{|c|} \sum_{\substack{d \in (\Z/c\Z)^* \\ M=\smat{a&b\\c&d} \in \SL_2(\Z)}} \ssprod{\e_\beta, \rho(\tilde M) \e_\gamma} \ e\left(\frac{ma+nd}{c}\right) .
\end{align}
The coefficients $\ssprod{\e_\beta, \rho(\tilde M) \e_\gamma}$ are universally bounded because $\rho$ factors through a double cover of the finite group $\SL_2(\Z/N\Z)$. Thus $H_c(\beta,m,\gamma,n)$ is bounded by some constant only depending on the underlying lattice $L$, and
the sum in (\ref{def Kloosterman zeta}) defines a holomorphic function in $s$ for $\Re(s)>1$.

\begin{proposition} \label{proposition FE of P}
	Let $\beta \in L'/L$ and $m \in \Z+Q(\beta)$. The Poincar\'e series $P_{\beta,m}(\tau,s)$ has a Fourier expansion of the form
	\begin{align*}
		P_{\beta,m}(\tau,s) = v^s e(m\tau) (\e_\beta + \e_{-\beta})
		+ b(0,0;v,s) \e_0
		+ \sum_{\gamma \in L'/L} \ \sum_{\substack{n \in \Z+Q(\gamma) \\ n \neq 0}} b(\gamma,n;v,s) e(nu) \e_\gamma ,
	\end{align*}
	for $\tau = u+iv \in \H$ and $s \in \C$ with $\Re(s)>3/4$. The Fourier coefficients $b(\gamma,n;v,s)$ are given by
	\[
		b(0,0;v,s) = \frac{2^{3/2-2s} \pi v^{1/2-s} }{\Gamma(s)} \ 
			\sum_{j=0}^\infty \frac{1}{j!} \left(-\frac{\pi m}{v}\right)^j
			\frac{ \Gamma(2s-1/2+j) }{ \Gamma(s+1/2+j) } \ 
			Z(1/4+s+j;\beta,m,0,0) ,
	\]
	if $n=0$, and by
	\begin{align*}
		b(\gamma,n;v,s)
		&= 2^{1/2} \pi^{s+1/2} |n|^{s-1/2}
			\sum_{j=0}^\infty \frac{(-4\pi^2|n|m)^j}{j!}
			Z(1/4+s+j;\beta,m,\gamma,n) \\
		&\hspace{1cm} \times \begin{dcases}
			\frac{(4\pi|n|v)^{-1/4-j/2}}{\Gamma(s+1/2+j)} \ W_{1/4+j/2, s-1/4+j/2} (4\pi|n|v) , & \text{if $n>0$} , \\
			\frac{(4\pi|n|v)^{-1/4-j/2}}{\Gamma(s)} \ W_{-1/4-j/2, s-1/4+j/2} (4\pi|n|v) , & \text{if $n<0$} ,
		\end{dcases}
	\end{align*}
	if $n \neq 0$.
\end{proposition}

\begin{proof}
	The proof proceeds by standard calculations, so we omit it. We refer the reader to \cite{DeAzevedoPribitkin1999}, Section 5, for some hints on the computation in the scalar valued case.
\end{proof}


\subsection{Eisenstein series of weight $1/2$} \label{m=0}

Since $N$ is squarefree, the only element $\beta \in L'/L$ with $Q(\beta) = 0 \mod \Z$ is $\beta = 0$. For $m = \beta = 0$ the Poincar\'e series
\[
	P_{0,0}(\tau,s) = E_{0}(\tau,s) 
\]
is an Eisenstein series. Note that the Fourier expansion given in Proposition \ref{proposition FE of P} greatly simplifies for $m = 0$ since the summands for $j > 0$ vanish. Further, one can rewrite the Kloosterman zeta functions $Z(1/4+s;0,0,\gamma,n)$ appearing in the Fourier expansion of $E_{0}(\tau,s)$ in terms of Dirichlet $L$-functions as in \cite{BruinierKuehn2003}, Theorem 3.3. More precisely, following the arguments in \cite{BruinierKuss2001}, Section 4, or \cite{BruinierKuehn2003}, Section 3, we find
\begin{align} \label{eq Z(0,n)}
	Z(1/4+s;0,0,\gamma,n) = \begin{dcases}
		\frac{\sqrt{2}}{\sqrt{N}}\frac{\zeta(4s-1)}{\zeta(4s)}\prod_{p\mid N}\frac{1+p^{1-2s}}{1+p^{-2s}} , & \text{if $n=0$} , \\
		\frac{\sqrt{2}}{\sqrt{N}}\frac{L(\chi_{D_{0}},2s)}{\zeta(4s)}\sigma_{\gamma,-n}(2s+1/2) , & \text{if $n \neq 0$} ,
	\end{dcases}
\end{align}
for $s \in \C$ with $\Re(s) > 3/4$, $\gamma \in L'/L$ and $n \in \Z+Q(\gamma)$. Here $D_{0}$ is a fundamental discriminant such that $D_0 f^2= 4N \ord(\gamma)^2 n$ for some $f \in \N$, $L(\chi_{D_{0}},s)$ is the Dirichlet $L$-function associated to the Kronecker symbol $\chi_{D_{0}} = \left(\frac{D_{0}}{\cdot} \right)$, and $\sigma_{\gamma,-n}(s)$ is the generalized divisor sum defined as follows (compare \cite{BruinierKuehn2003}, Theorem 3.3): We have
\[
	\sigma_{\gamma,-n}(2s+1/2) = \prod_{p \mid f^{2}D_{0}}\frac{1-\chi_{D_{0}}(p)p^{-2s}}{1-p^{-4s}}L_{\gamma,-n}^{(p)}(p^{-2s}) ,
\]
where $L_{\gamma,-n}^{(p)}(X)$ is a polynomial given by
\[
	L_{\gamma,-n}^{(p)}(X) = N_{\gamma,-n}(p^{w_{p}})X^{w_{p}} + (1-X)\sum_{\nu=0}^{w_{p}-1}N_{\gamma,-n}(p^{\nu})X^{\nu} \in \Z[X] .
\]
Here $w_{p} = 1 + 2v_{p}(2 \ord(\gamma) n)$ and
\[
	N_{\gamma,-n}(a) = \#\{x \in L/aL \colon Q(x-\gamma)-n \equiv 0 \mod a\}
\]
is a representation number.
Using this, we obtain a very explicit Fourier expansion of the Eisenstein series, from which we can derive the following statement:

\begin{proposition} \label{theorem MC of P}
	The Eisenstein series $P_{0,0}(\tau,s) = E_{0}(\tau,s)$ has a meromorphic continuation in $s$ to $\C$ which is given by its Fourier expansion, and which is holomorphic up to a simple pole at $s = 1/2$ and simple poles on the purely imaginary axis. In particular, it is holomorphic at $s = 0$.
\end{proposition}

\begin{proof}
	We need to study the generalized divisor sum $\sigma_{\gamma,-n}(2s+1/2)$ appearing in the Fourier expansion more closely. Firstly, we see that each $p$-factor of $\sigma_{\gamma,-n}(2s+1/2)$ with $\chi_{D_{0}}(p) = 1$ is holomorphic on $\C$ up to possible simple poles on the purely imaginary axis. It is easy to show that $N_{\gamma,-n}(p^{w_{p}}) \neq 0$ implies $\chi_{D_{0}}(p) = 1$. Thus, for $\chi_{D_{0}}(p) \in \{-1,0\}$ the polynomial $L_{\gamma,-n}^{(p)}(p^{-2s})$ has a root at $s = 0$ which compensates for the simple root of the denominator of the $p$-factor at $s = 0$. In particular, $\sigma_{\gamma,-n}(2s+1/2)$ is holomorphic on $\C$ up to simple poles on the purely imaginary axis, and it is holomorphic at $s=0$.
	
	The analytic properties of the remaining terms in the Fourier expansion of $E_{0}(\tau,s)$ are well known, and easily imply the claimed meromorphic continuation of $E_{0}(\tau,s)$.
\end{proof}

We note that the meromorphic continuation of the Eisenstein series also follows by the fundamental work of Selberg and Roelcke, see \cite{SelbergHarmonicAnalysis} and \cite{Roelcke1966,Roelcke1967}. The holomorphicity at $s=0$ is then implied by the functional equation of the Eisenstein series and the fact that its poles in the strip $1/4 < \Re(s) \leq 3/4$ are all simple.


\subsection{Continuation of Kloosterman zeta functions}

First of all, we establish the meromorphic continuation of the Poincar\'e series $P_{\beta,m}(\tau,s)$ for $m > 0$, using the spectral theory of vector valued modular functions developed in the extensive works \cite{Roelcke1966,Roelcke1967} of Roelcke. Adapting his results to the setting of this work, a combination of Satz 7.2 and Satz 12.3 in \cite{Roelcke1967} yields the following spectral theorem: Given a real analytic function $f \colon \H \to \C[L'/L]$ modular of weight $1/2$ with respect to the Weil representation $\rho$ and satisfying $(f,f) < \infty$, the function $f$ admits a spectral expansion of the form
\begin{align} \label{spectral theorem}
	f(\tau) = \sum_{j=0}^\infty (f,\psi_j) \psi_j(\tau)
		+ \frac{1}{64\pi} \int_{-\infty}^\infty (f,E_0(\,\cdot\,,1/4+it)) E_0(\tau,1/4+it) dt .
\end{align}
Here $(\psi_j)_{j \geq 0}$ is an orthonormal system of square-integrable eigenfunctions of the Laplace operator $\Delta_{1/2}$, modular of weight $1/2$ with respect to $\rho$, and $E_0(\tau,s) = P_{0,0}(\tau,s)$ is the unique vector valued Eisenstein series of weight $1/2$ for the given lattice $L$. The fudge factor $1/64\pi$ comes from our normalization of the Eisenstein series.

\begin{proposition} \label{proposition MC of P for m>0}
	Let $\beta \in L'/L$ and $m \in \Z+Q(\beta)$ with $m>0$. Then $P_{\beta,m}(\tau,s)$ has a meromorphic continuation in $s$ to $\C$ which is holomorphic at $s=0$.
\end{proposition}

\begin{proof}
	The following proof is a translation of Selbergs work in \cite{Selberg1965} to our situation. Some caution is necessary since in our case the Laplace operator $\Delta_{1/2}$ has a continuous spectrum. As in the classical setting the Poincar{\'e} series $P_{\beta,m}(\tau,s)$ is square-integrable for $m>0$, and thus admits a spectral expansion as in \eqref{spectral theorem}, i.e.,
	\begin{align} \label{eq spectral exp of P}
		P_{\beta,m}(\tau,s) = \sum_{j=0}^\infty a_j(s) \psi_j(\tau)
			+ \frac{1}{64\pi} \int_{-\infty}^\infty a_\infty(s,t) E_0\left(\tau,1/4+it\right) dt ,
	\end{align}
	with spectral coefficients given by
	\[
		a_j(s) = 2 (4\pi m)^{3/4-s} \ \frac{ \Gamma(s-1/4+it_j) \Gamma(s-1/4-it_j) }{ \Gamma(s) } \ \overline{c_{\psi_j}(\beta,m)}
	\]
	and
	\[
		a_\infty(s,t) = \frac{4^{it} \pi}{(4\pi m)^{s-1/4+it}} \ \frac{ \Gamma(s-1/4+it) \ \Gamma(s-1/4-it) }{ \Gamma(s) \ \Gamma(3/4-it) } \ \overline{ Z(1/2+it;0,0,\beta,m) } .
	\]
	Here the eigenfunction $\psi_j$ has the eigenvalue $\lambda_j = 1/16+t_j^2$ with $t_j>0$ or $t_j \in [-i/4,0]$. Further, $\psi_j$ has a Fourier expansion of the form
	\[
		\psi_j(\tau) = c_j(0,0) v^{1/4-it_j} \e_0 + \sum_{\gamma \in L'/L} \sum_{\substack{n \in \Z+Q(\gamma) \\ n \neq 0}} c_j(\gamma,n) v^{-1/4} W_{\sign(n)/4,it_j}(4\pi|n|v) e(nu) \e_\gamma .
	\]
	It is not difficult to see that the part in the spectral expansion in \eqref{eq spectral exp of P} coming from the discrete spectrum of $\Delta_{1/2}$ has a meromorphic continuation in $s$ to all of $\C$ which is holomorphic at $s=0$. In addition, we can establish the meromorphic continuation of the Eisenstein part in \eqref{eq spectral exp of P} by shifting the line of integration $\Re(w)=1/4$ to the left using the residue theorem (for details we refer to the proof of Theorem 4.2 in \cite{JorgensonKramerPippich2009}). The additional poles we encounter during this process are of the form $s=\xi-n$ and $s=1/2-\xi-n$ for $\xi$ a pole of the Eisenstein series $E_0(\tau,s)$ with $-1/4 \leq \Re(\xi) < 1/4$, and $n \in \N_0$. However, we have seen in Theorem \ref{theorem MC of P} that $\xi = 0$ is not a pole of the Eisenstein series, so the continuous part of the spectral expansion of $P_{\beta,m}(\tau,s)$ is also holomorphic at $s=0$. This finishes the proof.
\end{proof}

Next we use the previous proposition to conclude the meromorphic continuation of the Kloosterman zeta functions $Z(s;\beta,m,\gamma,n)$ as in \cite{Selberg1965}.

\begin{lemma} \label{lemma MC of Z}
	For $\beta, \gamma \in L'/L$ and $m \in \Z+Q(\beta)$, $n \in \Z+Q(\gamma)$ the Kloosterman zeta function $Z(s;\beta,m,\gamma,n)$ has a meromorphic continuation in $s$ to all of $\C$. If $m \geq 0$ or $n \geq 0$ this meromorphic continuation is holomorphic at $s=1/4$, and otherwise, i.e., if $m,n<0$, it has at most a simple pole at $s=1/4$.
\end{lemma}

\begin{proof}
	Firstly, let $m>0$ and $n \neq 0$. We consider the product $(P_{\beta,m}(\,\cdot\,,s),P_{\gamma,n}(\,\cdot\,,\overline{s}+\ell))$ for fixed $\ell \in \N$ and $\Re(s) \gg 0$. Unfolding against $P_{\gamma,n}(\tau,\overline{s}+\ell)$ and using the Fourier expansion of $P_{\beta,m}(\tau,s)$ one obtains an expression of the form
	\begin{align} \label{eq (Pm,Pn)}
		&(P_{\beta,m}(\,\cdot\,,s),P_{\gamma,n}(\,\cdot\,,\overline{s}+\ell)) \\
		&\hspace{0.5cm} = \sum_{j=0}^\ell A_\ell^{(j)}(s;m,n) Z(1/4+s+j;\beta,m,\gamma,n) + R_\ell(s;\beta,m,\gamma,n) . \notag
	\end{align}
	Here the functions $A_\ell^{(j)}(s;m,n)$ for $j=1,\hdots,\ell$ and $R_\ell(s;\beta,m,\gamma,n)$ are holomorphic in $s$ for $\Re(s)>-\ell/2$. Moreover, using the meromorphic continuation of $P_{\beta,m}(\tau,s)$ and the fact that both Poincar\'e series are square-integrable, we see that the left-hand side of \eqref{eq (Pm,Pn)} is meromorphic in $s$ for $\Re(s) \geq 3/4-\ell$. Therefore we inductively obtain the meromorphic continuation of $Z(s;\beta,m,\gamma,n)$ by considering \eqref{eq (Pm,Pn)} for $s \in \C$ with $\Re(s)>1/4-\ell/2$ for $\ell \in \N$. Moreover, evaluating \eqref{eq (Pm,Pn)} for $\ell=1$ at the point $s=0$, we find that $Z(s;\beta,m,\gamma,n)$ is holomorphic at $s=1/4$ since $P_{\beta,m}(\tau,s)$ is holomorphic at $s=0$ by Proposition \ref{proposition MC of P for m>0}.
	
	If $m<0$ and $n>0$ we can use the identity
	\begin{align} \label{eq identity for Z}
		Z(s;\beta,m,\gamma,n) = \overline{ Z(\overline{s};\gamma,n,\beta,m) }
	\end{align}
	to deduce the meromorphic continuation of the Kloosterman zeta function in this case, and its holomorphicity at the point $s=1/4$. For $m,n<0$ we have the identity
	\begin{align} \label{eq identity for Z*}
		Z(s;\beta,m,\gamma,n) = -Z^*(s;\gamma,-n,\beta,-m) ,
	\end{align}
	relating our Kloosterman zeta function to the Kloosterman zeta function $Z^*(s;\beta,m,\gamma,n)$ associated to the dual representation $\bar{\rho}$ of $\rho$. This dual Kloosterman zeta function is defined by replacing the Kloosterman sum in definition (\ref{def Kloosterman zeta}) by $H_c^*(\beta,m,\gamma,n)$, which is the Kloosterman sum associated to $\bar{\rho}$ of dual weight $k=2-1/2 = 3/2$, i.e.,
	\begin{align} \label{eq definition Hc*}
		H_c^*(\beta,m,\gamma,n) = \frac{e(-\sign(c)k/4)}{|c|} \sum_{\substack{d \in (\Z/c\Z)^* \\ M=\smat{a&b\\c&d} \in \SL_2(\Z)}} \ssprod{\rho(\tilde M) \e_\gamma, \e_\beta} \ e\left(\frac{ma+nd}{c}\right),
	\end{align}
	compare \cite{Bruinier2002}, eq. (1.13). These dual Kloosterman zeta functions appear in the Fourier expansion of
	 the non-holomorphic Poincar{\'e} series $P_{\gamma,-n}^*(\tau,s)$ of dual weight $k=3/2$ transforming with respect to the dual representation $\bar{\rho}$, i.e.
	\begin{align*}
		P_{\gamma,-n}^*(\tau,s)
		= \frac{1}{2} \sum_{(M,\phi) \in \ssprod{(T,1)} \backslash \Mp_2(\Z)} v^s e(-n\tau) \e_\gamma \weightkaction{3/2,\bar{\rho}} (M,\phi).
\end{align*}
	Using similar arguments as in the proof of Proposition \ref{proposition MC of P for m>0} we obtain the meromorphic continuation of this dual Poincar\'e series via its spectral expansion. Further, the analog of \eqref{eq (Pm,Pn)} looks like
	\begin{align}  \label{eq (Pm,Pn) dual}
		&(P_{\gamma,-n}^{*}(\,\cdot\,,s),P_{\beta,-m}^{*}(\,\cdot\,,\overline{s}+\ell)) \\
		&\hspace{0.5cm} = \sum_{j=0}^\ell \tilde{A}_\ell^{(j)}(s;-n,-m) Z^{*}(3/4+s+j;\gamma,-n,\beta,-m) + \tilde{R}_\ell(s;\gamma,-n,\beta,-m), \notag
	\end{align}
	yielding the meromorphic continuation of the dual Kloosterman zeta function. 
	
	In order to study the behaviour of $Z^{*}(s;\gamma,-n,\beta,-m)$ at $s = 1/4$ via  \eqref{eq (Pm,Pn) dual}, we need to investigate $P_{\gamma,-n}^{*}(\tau,s)$ at $s = -1/2$. Firstly, we note that following the arguments of Section \ref{m=0}, one finds that the dual Eisenstein series $P_{0,0}^*(\tau,s) = E_0^*(\tau,s)$ has a meromorphic continuation in $s$ to $\C$ with no poles on the real line. Thus, the continuous part of the spectral expansion of $P_{\gamma,-n}^*(\tau,s)$ does not contribute any relevant poles. However, if there is a holomorphic cusp form $\psi(\tau)$ of weight $3/2$ with respect to $\bar{\rho}$ whose $(\gamma,-n)$-th Fourier coefficient is non-zero, the discrete part of the spectral expansion, or more precisely this particular form, induces a simple pole of $P_{\gamma,-n}^*(\tau,s)$ at $s=-1/2$, with residue being essentially $\psi(\tau)$. Therefore, the dual Kloosterman zeta function $Z^*(3/4+s;\gamma,-n,\beta,-m)$ may have a simple pole at $s=-1/2$. So for $m,n<0$ the Kloosterman zeta function $Z(s;\beta,m,\gamma,n)$ has at most a simple pole at $s=1/4$ by the relation (\ref{eq identity for Z*}).
	
	Finally, we note that for $m=0$ the continuation of the Kloosterman zeta function and its holomorphicity at $s=1/4$ follow by the identity in \eqref{eq Z(0,n)} and the considerations in the proof of Theorem \ref{theorem MC of P}. The case $n=0$ is then dealt with using the identity \eqref{eq identity for Z}.
\end{proof}

In order to study the Fourier coefficients of $P_{\beta,m}(\tau,s)$ for $m \neq 0$ we also need to study the asymptotic behaviour of the Kloosterman zeta function $Z(s;\beta,m,\gamma,n)$ as $n \to \pm\infty$. Following the work of Pribitkin in \cite{Pribitkin2000} and adapting it to the present situation we obtain the following lemma.

\begin{lemma} \label{lemma growth of Z}
	Let $\beta,\gamma \in L'/L$, $m \in \Z+Q(\beta)$ with $m \neq 0$ and $\Omega \subseteq \C$ compact such that $Z(s;\beta,m,\gamma,n)$ is holomorphic for all $s \in \Omega$ and $n \in \Z+Q(\gamma)$. Then
	\begin{align} \label{eq estimate for Z}
		Z(s;\beta,m,\gamma,n) = O(|n|^\ell) \qquad \text{as $n \to \pm\infty$}
	\end{align}
	for some $\ell>0$, uniformly in $s$ for $s \in \Omega$. Here the implied constant depends on $\beta,m,\gamma$ and $\Omega$.
\end{lemma}

We quickly comment on the necessary adaptations of Pribitkin's work: Firstly, we need to translate his results to our vector valued setting. This is a tedious, but straightforward task. Secondly, Pribitkin only proves the above estimate for a multiplier system not allowing the existence of Eisenstein series, i.e., he assumes that the Laplace operator $\Delta_k$ has only discrete spectrum. This is not true in our case. However, this assumption is solely used in the proof of part (c) of Lemma 3 in his work, in order to derive good estimates for the Petersson norms $\snorm{P_{\beta,m}(\tau,s)}$ as $m \to \infty$. Since $m$ is fixed in our setting, we can simply estimate this single Petersson norm by taking the maximum over all $s \in \Omega$.


\subsection{Continuation to $s=0$}

We now establish the meromorphic continuation of the Poincar{\'e} series $P_{\beta,m}(\tau,s)$ for $m<0$. Since $P_{\beta,m}(\tau,s)$ is not square-integrable in this case, we cannot use spectral theory for the meromorphic continuation as in the case $m>0$. Instead we use the Fourier expansion of $P_{\beta,m}(\tau,s)$ given in Proposition \ref{proposition FE of P}. We also remark that in addition to proving the meromorphic continuation for $m<0$, we obtain that the known meromorphic continuation of $P_{\beta,m}(\tau,s)$ for $m>0$ (compare Proposition \ref{proposition MC of P for m>0}) is given by its Fourier expansion.

\begin{theorem} \label{theorem MC of P via FE}
	Let $\beta \in L'/L$ and $m \in \Z+Q(\beta)$. Then $P_{\beta,m}(\tau,s)$ has a meromorphic continuation in $s$ to $\C$ which is given by its Fourier expansion.
\end{theorem}

\begin{proof}
	We let $m \neq 0$ since for $m=0$ this has already been shown in Theorem \ref{theorem MC of P}. Further, we fix $\gamma \in L'/L$ and $\Omega \subseteq \C$ compact such that $Z(1/4+s+j;\beta,m,\gamma,n)$ is holomorphic for all $s \in \Omega$, $j \in \N_0$ and $n \in \Z+Q(\beta)$. Recall that apart from the harmless term $v^s e(m\tau) (\e_\beta+\e_{-\beta})$ the Fourier coefficients of $P_{\beta,m}(\tau,s)$ are all of the form
	\begin{align*}
		b(\gamma,n;v,s) = \sum_{j=0}^\infty c_j(n;v,s) Z(1/4+s+j;\beta,m,\gamma,n) .
	\end{align*}
	We choose $\ell \in \N_0$ such that $1/4+\Re(s)+\ell>1$ for all $s \in \Omega$, and split the above sum over $j$ into a finite sum $j=0,1,\hdots,\ell-1$ and an infinite sum starting with $j=\ell$. The finitely many terms of the first sum can now be estimated using Lemma \ref{lemma growth of Z} and the formula $W_{\kappa,\mu}(x) = O(x^\kappa e^{-x/2})$ as $x \to \infty$. In the second (infinite) sum the Kloosterman zeta functions can be bounded by a universal constant since the choice of $\ell$ guarantees that the real part of their argument is bigger than $1$.
	
	If $n=0$ then the remaining sum can be written as part of a confluent hypergeometric series $\HypF{1}{1}(a,b;x)$, which is meromorphic in $a$ and $b$. Thus we may assume that $n \neq 0$. In this case it remains to study the asymptotic behaviour of the Whittaker functions $W_{\pm 1/4 \pm j/2,s-1/4+j/2}(4\pi |n|v)$ in $j$, $n$ and $v$. Recall that
	\[
		W_{\kappa,\mu}(x) = e^{-x/2} x^{\mu+1/2} U(1/2+\mu-\kappa,1+2\mu,x) ,
	\]
	where $U(a,b,x)$ is Tricomi's confluent hypergeometric function. If $\Re(a)>0$ we can use the integral representation of $U(a,b,x)$ to obtain
	\begin{align} \label{eq estimate for U}
		\abs{ \Gamma(a) U(a,b,x) } \leq \begin{dcases}
			2^{\Re(b-a)-1} (x^{-\Re(a)} \Gamma(\Re(b)-1) + 1) , & \text{if $\Re(b-a)>1$} , \\
			x^{-\Re(a)} \Gamma(\Re(a)) , & \text{if $\Re(b-a) \leq 1$} .
		\end{dcases}
	\end{align}
	If $n<0$ then the condition $\Re(a)>0$ translates to $\Re(s)+1/2+j>0$, which is true for all $j \geq L$. Hence we may use the above estimate in this case. Putting everything carefully together, we find that
	\begin{align} \label{eq bound b(gamma,n,v,s) in n}
		b(\gamma,n;v,s) = O(e^{-\eps|n|}) \qquad \text{as $n \to -\infty$}
	\end{align}
	for some $\eps>0$, uniformly in $v$ and $s$, where $v$ is in a given compact subset $V$ of $\H$ and $s \in \Omega$. The implied constant depends on $\beta,m,\gamma,V$ and $\Omega$. Moreover, we find
	\begin{align} \label{eq bound b(gamma,n,v,s) in v}
		b(\gamma,n;v,s) = O(e^{-\delta v}) \qquad \text{as $v \to \infty$}
	\end{align}
	for some $\delta>0$, uniformly in $n$ and $s$, where $n<0$ and $s \in \Omega$. The implied constant depends on $\beta,m,\gamma$ and $\Omega$.
	
	If $n>0$ we repeatedly use the recurrence relation
	\[
		U(a,b,z) = zU(a+1,b+1,z) - (b-a-1)U(a+1,b,z)
	\]
	from \cite[eq. 13.4.18]{Abramowitz1984}, until $\Re(a+k) = \Re(s)+k$ is positive for all $s \in \Omega$. In each of these finitley many pieces appearing during this process, we can now use the estimate (\ref{eq estimate for U}), yielding the same bounds as in (\ref{eq bound b(gamma,n,v,s) in n}) and (\ref{eq bound b(gamma,n,v,s) in v}) for the coefficients $b(\gamma,n,v,s)$ with $n$ positive.
	
	Thus we have shown that also for $m \neq 0$ the Fourier series given in Proposition \ref{proposition FE of P} defines a meromorphic function in $s$ on $\C$ which is real analytic in $\tau$. This function is by construction the analytic continuation of $P_{\beta,m}(\tau,s)$ in $s$.
\end{proof}

\begin{lemma} \label{lemma P(tau,0)}
	Let $\beta \in L'/L$ and $m \in \Z+Q(\beta)$. The meromorphic continuation of $P_{\beta,m}(\tau,s)$ is holomorphic in $s=0$ with a Fourier expansion of the form
	\begin{align*}
		&P_{\beta,m}(\tau,0) = e(m\tau) (\e_\beta + \e_{-\beta})
			+ \sum_{\gamma \in L'/L} \sum_{\substack{n \in \Z+Q(\gamma) \\ n>0}} \tilde b(\gamma,n) e(n\tau) \e_\gamma \\
			&\hspace{3.5cm} + \sum_{\gamma \in L'/L} \sum_{\substack{n \in \Z+Q(\gamma) \\ n<0}} \tilde b(\gamma,n) \Gamma(1/2,4\pi|n|v) e(n\tau) \e_\gamma .
	\end{align*}
	Here the coefficients $\tilde b(\gamma,n)$ are real and do not depend on $v$ and $s$. If $m \geq 0$ then $\tilde b(\gamma,n)=0$ for all $\gamma \in L'/L$ and $n \in \Z+Q(\gamma)$ with $n<0$. We define
	\[
		P_{\beta,m}(\tau) = P_{\beta,m}(\tau,0) .
	\]
\end{lemma}

\begin{proof}
	First of all we note that by Lemma \ref{lemma MC of Z} the Kloosterman zeta functions $Z(1/4+s+j;\beta,m,\gamma,n)$ with $j \in \N_0$ showing up in the Fourier expansion are all holomorphic at $s=0$ except for the case $j=0$ and $m,n<0$ in which we may have a simple pole at that point. However, for $n < 0$ the factor $1/\Gamma(s)$ in the Fourier coefficient $b(\gamma,n;v,s)$ compensates for this possible pole if $m < 0$, and implies the vanishing of $b(\gamma,n;v,s)$ at $s = 0$ if $m \geq 0$. In particular, the Fourier coefficients of $P_{\beta,m}(\tau,s)$ are holomorphic at $s=0$. Using the formulas 
	\[
		W_{-1/4,-1/4}(x) = x^{1/4} e^{x/2} \Gamma(1/2,x)
		\qquad \text{and} \qquad
		W_{\mu,\mu-1/2}(x) = x^\mu e^{-x/2}
	\]
	for $x>0$ (see \cite{Erdelyi1955}, formulas (2),(21) and (36) in Section 6.9) we obtain an expansion of the claimed form. Moreover, the Fourier coefficients $\tilde b(\gamma,n)$ are indeed real since $\overline{Z(s;\beta,m,\gamma,n)} = Z(\overline{s};\beta,m,\gamma,n)$.
\end{proof}

\begin{theorem} \label{theorem P_beta,m(tau)}
	The function $P_{\beta,m}(\tau) \in H_{1/2,\rho}^+$  can be characterized as follows:
	\begin{enumerate}
		\item For $m > 0$ the function $P_{\beta,m}$ is the unique cusp form of weight $1/2$ for $\rho$ characterized by the inner product formula
		\begin{align} \label{PeterssonCoefficientFormula}
			(f,P_{\beta,m}) = -8\pi\sqrt{m} \, c_f(\beta,m)
		\end{align}
		for each $f = \sum_\gamma \sum_n c_f(\gamma,n) e(n\tau) \e_{\gamma} \in S_{1/2,\rho}$.
		\item For $m=\beta=0$ we have
		\[
			P_{0,0}(\tau) = \frac{2}{\sigma_{0}(N)}\sum_{c \mid N}\theta^{w_{c}} \in M_{1/2,\rho} ,
		\]
		where $\sigma_0(N) = \sum_{c \mid N} 1$.
		\item For $m < 0$, the function $P_{\beta,m}$ is the unique harmonic Maass form of weight $1/2$ for $\rho$ which has principal part $(\e_{\beta} + \e_{-\beta})e(m\tau)$, which is orthogonal to cusp forms with respect to the regularized inner product, and which maps to a cusp form under the differential operator $\xi_{1/2}$.
	\end{enumerate}
\end{theorem}


\begin{proof}
	Applying $\Delta_k$ to the Fourier expansion of $P_{\beta,m}$ given in Lemma \ref{lemma P(tau,0)}, or evaluating the differential equation (\ref{eq DE of P}) at $s=0$, we see that $P_{\beta,m}$ is harmonic. Moreover, $P_{\beta,m}$ is by analytic continuation modular of weight $1/2$ for $\rho$. Looking at its Fourier expansion we thus find that $P_{\beta,m}$ is an element of $S_{1/2,\rho}$, $M_{1/2,\rho}$ or $H_{1/2,\rho}^+$ if $m>0$, $m=0$ or $m<0$, respectively.
	
	If $m=\beta=0$ then $P_{0,0} \in M_{1/2,\rho}$ can be written as a linear combination of the unary theta series $\theta^{w_c}$ for $c \mid N$ introduced in Section \ref{section unary theta functions}, i.e., $P_{0,0} = \sum_{c \mid N} \lambda_{c} \theta^{w_{c}}$. On the other hand, $P_{0,0}$ is invariant under all Atkin-Lehner involutions, which can be checked for the defining series of $P_{0,0}(\tau,s)$ for $\Re(s) \gg 0$, and follows for $P_{0,0}(\tau)$ by analytic continuation. Hence all $\lambda_{c}$ agree. Comparing constant coefficients, we obtain $\lambda_{c} = 2/\sigma_{0}(N)$.
		
	Let now $m \neq 0$. The usual unfolding argument shows that for $f \in S_{1/2,\rho}$ and $\Re(s) \gg 0$ we have $(f,P_{\beta,m}(\cdot,s))^{\reg} = 0$ if $m<0$, and
	\[
		(f,P_{\beta,m}(\cdot,s))^{\reg} = 2(4\pi m)^{-s+1/2}\Gamma(s-1/2)c_{f}(\beta,m)
	\]
	if $m>0$. One can check that the regularized integral on the left-hand side has a meromorphic continuation to $s = 0$, and that its evaluation at $s=0$ agrees with $(f,P_{\beta,m})^{\reg}$. Hence $P_{\beta,m}$ is orthogonal to cusp forms if $m<0$, and satisfies the claimed formula if $m>0$.
	
	It remains to note that for $m<0$ the function $P_{\beta,m}$ is uniquely determined by the given conditions. Let $f \in H_{1/2,\rho}^{+}$ with $c_f^+(\gamma,n)=0$ for all $\gamma \in L'/L$ and $n<0$. Then
	\[
		(\xi_{1/2}f,\xi_{1/2}f) = \sum_{\gamma \in L'/L} \sum_{\substack{n \in \Z+Q(\gamma) \\ n<0}} c_{f}^{+}(\gamma,n) c_{\xi_{1/2}f}(\gamma,-n) = 0
	\]
	by Proposition 3.5 in \cite{BruinierFunke04}, and thus $\xi_{1/2}f = 0$, implying $f \in M_{1/2,\rho}$. Hence $P_{\beta,m}$ is uniquely determined by its principal part $(\e_\beta + \e_{-\beta}) e(m\tau)$ upto an element $f \in M_{1/2,\rho}$. Using that $c_{P_{\beta,m}}^+(0,0)=0$, and that $P_{\beta,m}$ is orthogonal to cusp forms, we thus obtain the claimed uniqueness.
\end{proof}

\begin{example}
	Let $N = 1$. Then we have $P_{0,0} = 2\theta$, and $P_{\beta,m} = 0$ for $m > 0$ since $S_{1/2,\rho} = \{0\}$ for $N = 1$. Further, for $m < 0$ the function $P_{\beta,m}$ is the unique weakly holomorphic modular form of weight $1/2$ with principal part $2e(m\tau)\e_{\beta}$. Under the isomorphism $f_{0}(\tau)\e_{0} + f_{1}(\tau)\e_{1} \mapsto f_{0}(4\tau) + f_{1}(4\tau)$ between $M_{1/2,\rho}^{!}$ and the space of scalar valued weakly holomorphic modular forms of weight $1/2$ for $\Gamma_{0}(4)$ satisfying the Kohnen plus space condition, $P_{\beta,m}$ is identified with Borcherds' basis element $2f_{4m} = 2q^{4m} + O(q)$, whose Fourier coefficients of positive index are given by twisted traces of the $j$-function (compare \cite{ZagierTraces}, Section 5).
\end{example}
\section{Eisenstein series as theta lifts} \label{section lifts}

In the current section we realize averaged versions of the parabolic, hyperbolic and elliptic Eisenstein series as regularized theta lifts of Selberg's Poincar\'e series introduced in Section \ref{section Poincare series}. We start by investigating the analytic properties of this lift. For $\beta \in L'/L$ and $m \in \Z+Q(\beta)$ with $m<0$ we define $H_{\beta,m}$ as the set of all Heegner points, i.e.
\begin{align*}
	H_{\beta,m} = \{z_X \in \H \colon X \in L_{\beta,m}\} .
\end{align*}
If $m \geq 0$ we simply put $H_{\beta,m} = \emptyset$.

\begin{proposition}\label{proposition theta lift continuation}
	Let $\beta \in L'/L$ and $m \in \Z + Q(\beta)$. 
	\begin{enumerate}
		\item For $s \in \C$ with $\Re(s) > 1/2$ the regularized theta lift
	\[
		\Phi(z,P_{\beta,m}(\,\cdot\,,s)) = \CT_{t=0} \left[ \lim_{T \to \infty} \int_{\calF_T} \sprod{P_{\beta,m}(\tau,s),\Theta(\tau,z)} \Im(\tau)^{1/2-t} \frac{du \, dv}{v^{2}} \right]
	\]
	defines a real analytic function in $z \in \H \setminus H_{\beta,m}$ and a holomorphic function in $s$, which can be meromorphically continued to all $s \in \C$.
		\item For $m \neq 0$, it is holomorphic at $s = 0$, and we have
		\begin{align*}
		\Phi(z,P_{\beta,m}(\,\cdot\,,0)) = \Phi(z,P_{\beta,m}(\,\cdot\,,s)) \big|_{s = 0}.
		\end{align*} 
		\item For $m = \beta = 0$, it has a simple pole at $s = 0$ with residue $-\frac{2}{s}$, and we have
		\begin{align*}
		\Phi(z,P_{0,0}(\,\cdot\,,0)) = \left(\Phi(z,P_{0,0}(\,\cdot\,,s)) + \frac{2}{s} \right)\bigg|_{s = 0}.
		\end{align*}
	\end{enumerate}
\end{proposition}

\begin{proof}
	The integral of $P_{\beta,m}(\tau,s)$ against $\Theta(\tau,z)$ over the compact set $\calF_1$, the fundamental domain truncuated at $y=1$, converges and is real analytic for all $z \in \H$, and holomorphic in $s$ whenever $P_{\beta,m}(\tau,s)$ is. Thus it suffices to consider
	\[
		\vphi(z,s,t) = \int_1^\infty \int_{-1/2}^{1/2} \sprod{P_{\beta,m}(\tau,s),\Theta(\tau,z)} v^{1/2-t} \frac{du \, dv}{v^2}.
	\]
	Plugging in the Fourier expansion of $P_{\beta,m}(\tau,s)$ and the defining series for $\Theta(\tau,z)$, and carrying out the integral over $u$, we find
	\begin{align} \label{eq proof convergence of lift 1}
		\vphi(z,s,t) &=
			2\int_1^\infty \sum_{X \in L_{\beta,m}} v^{s-1-t} \exp(-4\pi vQ(X_{z})) dv+ \int_1^\infty b(0,0;v,s) v^{-1-t} dv \notag  \\
			&\quad + \int_1^\infty \sum_{\substack{X \in L' \\ X \neq 0}} b(X,Q(X);v,s) v^{-1-t} \exp(-2\pi vQ_z(X)) dv,
	\end{align}
	where $Q_z(X) = Q(X_{z})-Q(X_{z^{\perp}})$ is the positive definite majorant of $Q$ associated to $z$. If $m = \beta =0$, we further split off the summand for $X =0$ in the first integral, giving the additional term
	\begin{align}\label{eq bad integral}
		2\int_{1}^{\infty}v^{s-1-t}dv = -\frac{2}{s-t}
	\end{align}
	for $\Re(t) > \Re(s)$. For $\Re(s) > 0$, the right-hand side has a continuation to $t = 0$, and taking the constant term at $t = 0$ yields a simple pole of $\Phi(z,P_{0,0}(\cdot,s))$ at $s = 0$ with residue $-\frac{2}{s}$. 
	
	The simple estimate $b(0,0;v,s) = O(v^{1/2-s})$ shows that for $\Re(s) > 1/2$ the second integral in \eqref{eq proof convergence of lift 1} is holomorphic at $t = 0$. We plug in $t = 0$, insert the explicit formula for $b(0,0;v,s)$ and evaluate the integrals over the powers of $v$, yielding
	\begin{align} \label{eq proof convergence of lift 3}
		\frac{2^{3/2-2s} \pi}{\Gamma(s)} \ 
			\sum_{j=0}^\infty \frac{(-\pi m)^j}{j!}
			\frac{ \Gamma(2s-1/2+j) }{ \Gamma(s+1/2+j) } \ 
			Z(1/4 + s+j;\beta,m,0,0)\frac{1}{s+j-1/2} .
	\end{align}
	By Lemma \ref{lemma MC of Z} the remaining expression has a meromorphic continuation in $s$ to $\C$, which is holomorphic and indeed vanishing at $s = 0$.
	
	Next we recall from the proof of Theorem \ref{theorem MC of P via FE} that for $n \neq 0$ the Fourier coeffcients $b(\gamma,n;v,s)$ are rapidly decreasing in $n$ and $v$ (see (\ref{eq bound b(gamma,n,v,s) in n}) and (\ref{eq bound b(gamma,n,v,s) in v})). Hence it can be shown as in the proof of Proposition 2.8 in \cite{Bruinier2002} that for $z \in \H \setminus H_{\beta,m}$ each of the remaining integrals in \eqref{eq proof convergence of lift 1} has a continuation to $t=0$, which is real analytic in $z \in \H \setminus H_{\beta,m}$ and has a continuation to all $s \in \C$ for which $P_{\beta,m}(\tau,s)$ is holomorphic. In particular, we can just plug in $t=s=0$ in these integrals.
	
	In order to prove the equations in part (2) and (3) of the proposition, we have to go through the same proof again, replacing $P_{\beta,m}(\tau,s)$ by $P_{\beta,m}(\tau,0)$. Using similar arguments as before one can show that all the integrals appearing during this process apart from the one corresponding to \eqref{eq bad integral} have holomorphic continuation to $t=0$. This proves part (2). For $m = \beta = 0$, the integral corresponding to \eqref{eq bad integral} is now given by
	\begin{align}\label{eq bad integral 2}
		2\int_{1}^{\infty}v^{-1-t} = \frac{2}{t}
	\end{align}
	for $\Re(t) > 0$, with vanishing constant term at $t = 0$. Comparing the constant terms of \eqref{eq bad integral} and \eqref{eq bad integral 2} at $t = 0$, we obtain part (3) of the proposition.
\end{proof}

Recall that given $X \in L_{\beta,m}$ with $m \neq 0$ there is an associated Heegner geodesic $c_X$ or Heegner point $z_X$ if $m>0$ or $m<0$, respectively. Moreover, given $X \in L_{0,0}$ we find $\lambda \in \Z$ such that $X=\lambda X_\ell$ where $X_\ell$ is a generator of the isotropic line $\ell = \Q X$ satisfying $\ell \cap L = \Z X_\ell$, and there is an associated cusp $p_X$ corresponding to the line $\ell$.

\begin{lemma} \label{lemma Q(X_zperp)}
	Let $X \in L_{\beta,m}$ with $\beta \in L'/L$ and $m \in \Z + Q(\beta)$, and let $z \in \H$. For $m \neq 0$ we have
	\[
		Q(X_{z}) = \begin{dcases}
			m \cosh^2(\dhyp(z,c_X)) , & \text{if $m>0$} , \\
			|m| \sinh^2(\dhyp(z,z_X)) , & \text{if $m<0$},
		\end{dcases}
	\]
	and for $m = \beta = 0$ and $X \neq 0$ we have
	\[
		Q(X_{z}) = \frac{\lambda^2}{4N} \Im(\sigma_{p_X}^{-1} z)^{-2} ,
	\]
	were $\lambda \in \Z$ with $X = \lambda X_\ell$ for a generator $X_{\ell}$ of $\Q X \cap L$, and $\sigma_{p_X} \in \SL_2(\R)$ is a scaling matrix for the cusp $p_X$.
\end{lemma}

\begin{proof}
	The two formulas for $m \neq 0$ are well known and follow by a direct calculation. Let $m = \beta =0$, and let $c$ be a positive divisor of $N$. Then the cusp $1/c$ of $\Gamma_0(N)$ corresponds to the isotropic line $\Q X_c$ with generator $X_c = \smat{-1&1/c\\-c&1} \in L$, and a simple calculation shows that
	\[
		Q((X_{c})_{z}) = \frac{1}{4N} \Im(\sigma_{1/c}^{-1}z)^{-2}
		\qquad \text{with} \qquad
		\sigma_{1/c} = \sqrt{\frac{c}{N}} \pmat{N/c & 1 \\ N & 1+c} .
	\]
	Here $\sigma_{1/c}$ is a scaling matrix for the cusp $1/c$. Since $Q((\lambda X)_{z}) = \lambda^2 Q(X_{z})$ for $\lambda \in \Z$, and since the identification in (\ref{eq cusps and isotropic lines}) is compatible with the action of $\Gamma_0(N)$, this proves the claimed formula.
\end{proof}

\begin{theorem} \label{theorem lift 1}
	Let $\beta \in L'/L$ and $m \in \Z + Q(\beta)$. Then
	\[
		\Phi(z,P_{\beta,m}(\,\cdot\,,s)) = \begin{dcases}
			\frac{2 \Gamma(s)}{(4\pi m)^s} \sum_{X \in \Gamma_0(N)\setminus L_{\beta,m}} \Ehyp_{c_X}(z,2s) , & \text{if $m>0$} , \\
			4 N^s \zeta^*(2s) \sum_{p \in C(\Gamma_0(N))} \Epar_{p}(z,2s) , & \text{if $m=0$} , \\
			\frac{2 \Gamma(s)}{(4\pi|m|)^s} \sum_{X \in \Gamma_0(N)\setminus L_{\beta,m}} \Eell_{z_X}(z,2s) , & \text{if $m<0$} ,
		\end{dcases}
	\]
	for $z \in \H \setminus H_{\beta,m}$ and $s \in \C$ with $\Re(s) > 1/2$. Here, $\zeta^*(s) = \pi^{-s/2} \Gamma(s/2) \zeta(s)$ is the completed Riemann zeta function.
\end{theorem}

\begin{proof}
	Splitting the sum over $(M,\phi) \in \ssprod{(T,1)} \backslash \Mp_2(\Z)$ defining the Poincar\'e series into matri\-ces $M$ with lower left entry $c=0$ and $c \neq 0$, we find
	\begin{align} \label{proof of theorem lift 1, eq 1}
		P_{\beta,m}(\tau,s)
			= v^s e(m\tau) \left( \e_\beta + \e_{-\beta} \right)
			+ \frac{1}{2} \sum_{\substack{ (M,\phi) \in \ssprod{(T,1)} \backslash \Mp_2(\Z) \\ M \neq \pm1 }} v^s e(m\tau) \e_\beta \weightkaction{1/2,\rho} (M,\phi) .
	\end{align}
	As the infinite sum behaves nicely at $i\infty$, we may plug in $t = 0$ and take the limit in the corresponding part of the regularized integral, yielding
	\begin{align*}
		\int_{\calF} \sprod{ \frac{1}{2} \sum_{M \neq \pm1} v^s e(m\tau) \e_\beta \weightkaction{1/2,\rho} \tilde M , \Theta(\tau,z)} v^{1/2} \frac{du \,dv}{v^{2}}
		= 2 \int_{\calG} v^{s+1/2} e(m\tau) \overline{\Theta_\beta(\tau,z)} \ \frac{du \, dv}{v^{2}},
	\end{align*}
	where $\calG = \{\tau\in\H \colon |\Re(\tau)| \leq 1/2, |\tau| < 1\}$ and $\Theta = \sum_{\gamma \in L'/L} \Theta_\gamma \e_\gamma$. The asymptotic behavior $\Theta_{\beta}(\tau,z) = O(v^{-1})$ as $v \to 0$, uniformly in $u$ (see \cite{Bruinier2002}, Lemma 2.13), shows that the unfolding is justified for $\Re(s)>3/2$.
	If we split the regularized integral corresponding to the first part of \eqref{proof of theorem lift 1, eq 1} at $v = 1$, we see that $\Phi(z,P_{\beta,m}(\,\cdot\,,s))$ equals
	\begin{align*}
		2 \int_{0}^{1}\int_{-1/2}^{1/2} v^{s+1/2} e(m\tau) \overline{\Theta_\beta(\tau,z)} \ \frac{du \, dv}{v^{2}}
			+ \CT_{t=0}  2 \int_{1}^{\infty}\int_{-1/2}^{1/2} v^{s+1/2-t} e(m\tau) \overline{\Theta_\beta(\tau,z)} \ \frac{du \, dv}{v^{2}} .
	\end{align*}
	Plugging in the definition of $\Theta_\beta(\tau,z)$ and evaluating the integral over $u$ this becomes
	\begin{align*}
		2\int_0^1 v^{s-1} \sum_{X \in L_{\beta,m}} \exp(-4\pi vQ(X_{z})) dv 
		+ \CT_{t = 0}2\int_1^\infty v^{s-1-t} \sum_{X \in L_{\beta,m}} \exp(-4\pi vQ(X_{z})) dv .
	\end{align*}
	For $m = \beta = 0$ a short calculation shows that the integrals over the summands for $X = 0$ in the two sums above cancel out. Thus we may assume $X \neq 0$ in both sums. Now for $z \notin H_{\beta,m}$ and $X \in L_{\beta,m}\setminus \{0\}$ we have $Q(X_{z}) \neq 0$, so we can simply plug in $t = 0$. Substituting $v' = 4\pi vQ(X_{z})$ we obtain
	\begin{align} \label{eq generalized Eisenstein series?}
	\Phi(z,P_{\beta,m}(\,\cdot\,,s)) = \frac{2 \, \Gamma(s)}{(4\pi)^s} \sum_{X \in L_{\beta,m}\setminus\{0\} } Q(X_{z})^{-s}.
	\end{align}
	Next we split the sum in \eqref{eq generalized Eisenstein series?} into a sum over classes $[X]$ in $\Gamma_0(N) \backslash L_{\beta,m} \setminus \{0\}$ and a sum over matrices $M$ in $\Gamma_0(N)_X \backslash \Gamma_0(N)$ where $\Gamma_0(N)_X$ is the stabilizer of $X$ in $\Gamma_0(N)$. Since $Q((MX)_{z})=Q(X_{M^{-1}z})$, we find
	\[
		\Phi(z,P_{\beta,m}(\,\cdot\,,s)) = \frac{2 \, \Gamma(s)}{(4\pi)^s} \sum_{X \in \Gamma_0(N) \backslash (L_{\beta,m}\setminus\{0\})} \ \sum_{M \in \Gamma_0(N)_X \backslash \Gamma_0(N)} Q(X_{Mz})^{-s} .
	\]
	Note that $\Gamma_0(N)_X = \Gamma_0(N)_{c_X}$ if $m>0$, and $\Gamma_0(N)_X = \Gamma_0(N)_{z_X}$ if $m<0$. Hence for $m \neq 0$ the statement of the theorem follows from Lemma \ref{lemma Q(X_zperp)} and using holomorphic continuation of both sides to $\Re(s) > 1/2$.
	Let now $m=\beta=0$. Given $X \in L_{0,0} \setminus \{0\}$ we write $X = \lambda X_\ell$ with $\lambda \in \Z$ and $X_\ell$ a generator of $\Q X \cap L$. Note that $\Gamma_0(N)_X = \Gamma_0(N)_{X_\ell} = \Gamma_0(N)_p$ with $p$ being the cusp associated to the isotropic line $\Q X$. Applying Lemma \ref{lemma Q(X_zperp)} we obtain
	\begin{align*}
		\Phi(z,P_{\beta,m}(\,\cdot\,,s)) &= \frac{2 \, \Gamma(s)}{(4\pi)^s} \ \sum_{\substack{\lambda \in \Z \\ \lambda \neq 0}} \ \sum_{\ell \in \Gamma_0(N) \backslash \Iso(V)} \ \sum_{M \in \Gamma_0(N)_{X_\ell} \backslash \Gamma_0(N)} Q((\lambda X_\ell)_{Mz})^{-s} \\
		&= \frac{4 N^s \Gamma(s) \zeta(2s)}{\pi^s} \sum_{p \in C(\Gamma_0(N))} \ \sum_{M \in \Gamma_0(N)_{p} \backslash \Gamma_0(N)} \Im(\sigma_{p}^{-1} Mz)^{2s}.
	\end{align*}
	Here $\sigma_p$ is a scaling matrix for the cusp $p$. As before the claimed statement follows by holomorphic continuation.
\end{proof}

We conclude this section with a simple corollary on the meromorphic continuation of Eisenstein series which follows directly by Proposition \ref{proposition theta lift continuation}. Further, we remark a possible generalization of averaged hyperbolic and elliptic Eisenstein series.

\begin{corollary}
	The averaged Eisenstein series appearing in Theorem \ref{theorem lift 1} have a meromorphic continuation in $s$ to $\C$. 
\end{corollary}

\begin{remark}
	Upto equation \eqref{eq generalized Eisenstein series?} we did not use the special shape of the lattice $L$. In fact, given an even lattice $(L,Q)$ of arbitrary signature $(b^{+},b^{-})$ and some pair $(\beta,m)$ with $\beta \in L'/L$, $m \in \Z+Q(\beta)$ one could use the right-hand side of \eqref{eq generalized Eisenstein series?} to define an associated averaged Eisenstein series, which we call hyperbolic, parabolic or elliptic if $m>0$, $m=0$ or $m<0$, respecitvely. This generalized Eisenstein series lives on the Grassmannian of $b^{+}$-dimensional positive definite subspaces of $L \otimes \R$, and it is invariant under the action of $O_d(L)$, which is the subgroup of the special orthogonal group of $L$ fixing the classes of $L$ in its discriminant group $L'/L$.
	Moreover, using the same techniques as above an analog of equation \eqref{eq generalized Eisenstein series?} shows that these generalized Eisenstein series can again be realized as the Borcherds theta lift of Selberg's Poincar\'e series for the lattice $L$ of index $(\beta,m)$.
\end{remark}

	\section{Kronecker limit formulas}\label{section Kronecker limit functions}
	
	We now compute Kronecker limit formulas for the averaged parabolic, hyperbolic and elliptic Eisenstein series given in Theorem \ref{theorem lift 1}. Since $\Gamma(s)$ and $\zeta^{*}(2s)$ have simple poles at $s = 0$, Theorem \ref{theorem lift 1} tells us that the linear coefficient of the Laurent expansion of the averaged Eisenstein series at $s=0$ is essentially given by the value of the regularized theta lift $\Phi(z,P_{\beta,m}(\,\cdot\,,s))$ at $s = 0$. By Proposition \ref{proposition theta lift continuation} we need to compute the regularized theta lift of the harmonic Maass form $P_{\beta,m}(\tau) = P_{\beta,m}(\tau,0)$, which by Theorem \ref{theorem Borcherds products} is given by the absolute value of the logarithm of the Borcherds product $\Psi(z,P_{\beta,m})$ associated to $P_{\beta,m}(\tau)$. In the following, we compute these Borcherds products explicitly.
		
	\subsection{The parabolic case}
	
	We start with the classical, parabolic case. Let $m = \beta = 0$. In Theorem \ref{theorem P_beta,m(tau)} we have seen that
	\[
	P_{0,0}(\tau) = \frac{2}{\sigma_{0}(N)}\sum_{c \mid N}\theta^{w_{c}}(\tau)
	\]
	is a holomorphic modular form of weight $1/2$ for $\rho$. In order to compute the Borcherds product associated to $P_{0,0}$ we first compute the Borcherds product of the unary theta functions $\theta^{w_c}$.
	
	\begin{lemma} \label{lemma Borcherds products theta}
		For $c \mid N$ the Borcherds product associated to $\theta^{w_{c}}$ is given by
		\[
			\Psi(z,\theta^{w_{c}}) = \eta(cz)\eta\left(\tfrac{N}{c}z\right),
		\]
		where $\eta(z) = e(z/24)\prod_{n\geq 1}(1-e(nz))$ is the Dedekind eta function.
	\end{lemma}
	
	\begin{proof}
		The Fourier coefficients of the unary theta function $\theta^{w_{c}}$ relevant for its Borcherds product are given by
		\begin{align} \label{eq Fourier coefficients of theta_sigma_c}
			c_{\theta^{w_{c}}}(n,n^{2}/4N) =
			\begin{cases}
				1, & \text{if }n \equiv 0 \mod c \text{ or } n \equiv 0 \mod N/c, \text{ and } n \not \equiv 0 \mod N ,\\
				2, & \text{if }n \equiv 0 \mod N,\\
				0, & \text{else} ,
			\end{cases}
		\end{align}
		for $n \neq 0$. Further, the Weyl vector associated to $\theta^{w_c}$ and the cusp $\infty$ computes to
		\[
			\rho_{\theta^{w_{c}},\infty} = \frac{\sqrt{N}}{8\pi} (\theta^{w_c},\theta^{w_N})^{\reg} = \frac{c+N/c}{24}
		\]
		by Example 5.6 in \cite{BruinierSchwagenscheidt}. Therefore, using Theorem \ref{theorem Borcherds products} we obtain the product expansion
		\[
			\Psi(z,\theta^{w_{c}}) = e((c+N/c)z/24) \prod_{n=1}^\infty \big(1-e(ncz)\big)\prod_{n=1}^\infty \big(1-e\left(n\tfrac{N}{c}z\right)\big) .
		\]
		This proves the claimed formula.
	\end{proof}
	
	\begin{theorem}\label{theorem parabolic KLF}
		At $s=0$ we have the Laurent expansion
		\begin{align*}
			\sum_{p \in C(\Gamma_{0}(N))} \Epar_p(z,s) 
			= 1+\frac{1}{\sigma_{0}(N)}\sum_{c \mid N}\log\left(\left| \Delta(cz)\right|^{1/6}\Im(z)\right)\cdot s + O(s^{2}) .
		\end{align*}
	\end{theorem}
	
	\begin{proof}
		For brevity, we write $E(z,s)$ for the sum of the parabolic Eisenstein series on the left-hand side. By Proposition \ref{proposition theta lift continuation} we have the identity
		\begin{align*}
			\Phi(z,P_{0,0}) = \left( \Phi(z,P_{0,0}(\,\cdot\,,s)) + \frac{2}{s} \right) \bigg|_{s=0}
		\end{align*}
		of regularized theta lifts. Applying Theorem \ref{theorem Borcherds products} with $c_{P_{0,0}}^+(0,0)=2$ on the left-hand side and Theorem \ref{theorem lift 1} on the right-hand side, where we also replace $s$ by $s/2$, we obtain
		\begin{align}\label{ThetaLiftEquationParabolic}
			- 2(\log(4\pi N)+\Gamma'(1)) - 4\log|\Psi(z,P_{0,0})y|
			= \left( 4N^{s/2}\zeta^{*}(s)E(z,s)+\frac{4}{s} \right) \bigg|_{s=0} .
\end{align}
		Using the expansion $\zeta^{*}(s) = -\frac{1}{s} - \frac{1}{2}(\Gamma'(1) + \log(4\pi)) + O(s)$ and the fact that the right-hand side of \eqref{ThetaLiftEquationParabolic} is holomorphic at $s = 0$, we see that the constant term in the Laurent expansion of $E(z,s)$ at $s = 0$ equals $1$. Now a short calculation shows that the right-hand side of \eqref{ThetaLiftEquationParabolic} is given by
		\[
			- 2(\log(4\pi N)+\Gamma'(1)) - 4\,\FT_{s=0} E(z,s) ,
		\]
		where $\FT_{s=0} E(z,s)$ denotes the first term in the Laurent expansion of $E(z,s)$ at $s = 0$, i.e., the coefficient of the linear term. So by \eqref{ThetaLiftEquationParabolic} we find that $E(z,s)$ has the Laurent expansion
		\[
			E(z,s) = 1 + \log|\Psi(z,P_{0,0})y| \cdot s + O(s^2)
		\]
		at $s=0$. In order to determine the remaining Borcherds product, we write $P_{0,0}$ as a linear combination of theta functions $\theta^{w_{c}}$ for $c \mid N$ as in Theorem \ref{theorem P_beta,m(tau)} and apply Lemma \ref{lemma Borcherds products theta}. This yields
		\[
		\Psi(z,P_{0,0}) = \prod_{c \mid N}\left(\eta(cz)\eta\left(\tfrac{N}{c}z\right) \right)^{2/\sigma_{0}(N)} = \prod_{c \mid N}\eta(cz)^{4/\sigma_{0}(N)}.
		\]
		Using $\Delta(z) = \eta(z)^{24}$ and taking everything together we obtain the stated formula.
	\end{proof}
		
	\begin{remark}
		For $N=1$ we recover the classical Kronecker limit formula
		\[
			\Epar_\infty(z,s) = 1 + \log\left(|\Delta(z)|^{1/6}\Im(z)\right) \cdot s + O(s^2) .
		\]
		Moreover, for the extension $\Gamma_0^*(N)$ of $\Gamma_0(N)$ by all Atkin-Lehner involutions, having only one cusp at $\infty$, the sum of all parabolic Eisenstein series for $\Gamma_0(N)$ is actually the unique parabolic Eisenstein series for $\Gamma_0^*(N)$, i.e.,
		\[
			E_\infty^{\parabolic,\Gamma_0^*(N)}(z,s)
			= \sum_{M \in \Gamma_0^*(N)_\infty \backslash \Gamma_0^*(N)} \Im(Mz)^s
			= \sum_{p \in C(\Gamma_{0}(N))} \Epar_p(z,s) .
		\]
		Therefore, the above theorem actually states the Kronecker limit formula for the group $\Gamma_0^*(N)$, given for example in Section 1.5 of \cite{JorgensonSmajlovicThen}.
	\end{remark}
	
	\subsection{The hyperbolic case}
	
	Let $m > 0$ and $\beta \in L'/L$ with $m \in Q(\beta)+\Z$. Recall from Theorem \ref{theorem P_beta,m(tau)} that in this case the Poincar\'e series $P_{\beta,m}(\tau) = P_{\beta,m}(\tau,0)$ is the holomorphic Poincar\'e series of weight $1/2$ for $\rho$, i.e., the unique cusp form satisfying the Petersson coefficient formula \eqref{PeterssonCoefficientFormula}.
	
	We will first construct an orthogonal basis for the space $S_{1/2,\rho}$ consisting of linear combinations of unary theta functions, whose associated Borcherds products were computed in Lemma \ref{lemma Borcherds products theta}. This will then enable us to compute the Borcherds product associated to the Poincar\'e series $P_{\beta,m}(\tau)$.
	
	Note that we can identify the orthogonal group $O(L'/L)$ with the group $D(N)$ given by the set of positive divisors of $N$ together with the operation $c \star d = cd/(c,d)^{2}$ via the isomorphism $c \mapsto w_c$, where $w_c \in O(L'/L)$ is defined by \eqref{AtkinLehnerInvolution}. The characters of $D(N)$ are given by the functions $c \mapsto \mu((c,d))$ for $d \mid N$, where $\mu(n)$ denotes the Moebius function. We let $E(N)$ be the set of all $d \mid N$ for which $\mu(d) = 1$, and for $d \in E(N)$ we define
	\begin{align} \label{definition theta_d}
		\theta_{d} = \sum_{\substack{c \mid N}}\mu((c,d))\theta^{w_{c}} .
	\end{align}
	
	\begin{lemma} \label{lemma theta_d}
		The functions $\theta_d$ for $d \in E(N) \setminus \{1\}$ form an orthogonal basis of $S_{1/2,\rho}$ with norms
		\[
			(\theta_{d},\theta_{d})
			= \frac{2\pi \sigma_{0}(N)\sigma_{1}(N/d)\varphi(d)}{3\sqrt{N}} .
		\]
		Here $\sigma_{k}(N) = \sum_{d \mid N}d^{k}$, and $\varphi$ denotes Euler's totient function.
	\end{lemma}
	
	\begin{proof}
		Since $(f^{w_c},g) = (f,g^{w_c})$ one easily computes
		\begin{align} \label{eq lemma theta_d 1}
			(\theta_d,\theta_e)
			= \sum_{b,c \mid N} \mu((c,d)) \mu((b,e)) (\theta^{w_c},\theta^{w_e})
			= \sum_{c \mid N} \mu((c,d \star e)) \cdot \sum_{a \mid N} \mu((a,e)) (\theta^{w_a},\theta)
		\end{align}
		for $d,e \in E(N)$. Thus, using the orthogonality relations for characters of finite groups, we see that the functions $\theta_d$ for $d \in E(N)$ are orthogonal. Further, by the formula \eqref{eq Fourier coefficients of theta_sigma_c}, for $b \mid N$ the $(N/b,N^2/b^2 4N)$-Fourier coeffcient of $\theta_d$ is non-zero if and only if $d=b$. Hence, the $\theta_d$ are linearly independent, and since $M_{1/2,\rho}$ is isomorphic to the space $J_{1,N}^{*}$ of skew-holomorphic Jacobi forms of weight $1$ and index $N$ (compare \cite{EichlerZagier}, Theorem 5.7), the dimension formula $\dim(J_{1,N}^{*}) = \frac{1}{2}\sigma_{0}(N)$ from \cite[page 130]{SkoruppaZagier} shows that the $\theta_{d}$ for $d \in E(N)$ form a basis of $M_{1/2,\rho}$. Considering the constant coefficient of $\theta_d$ one easily checks that $\theta_d$ is a cusp form unless $d=1$. Hence the $\theta_d$ for $d \in E(N) \setminus \{1\}$ form an orthogonal basis of $S_{1/2,\rho}$.
		
		It remains to compute their norms. Recall that we have $(\theta^{w_{c}},\theta) = \frac{\pi}{3 \sqrt{N}} (c+N/c)$ for $c \mid N$ by Example 5.6 in \cite{BruinierSchwagenscheidt}. Thus \eqref{eq lemma theta_d 1} yields
		\[
			(\theta_d,\theta_d) = \sigma_0(N) \cdot \frac{2\pi}{3\sqrt{N}} \sum_{c \mid N} \mu((c,d)) c
		\]
		for $d \in E(N)$. Since
		\[
			\sum_{c \mid N} \mu((c,d)) c = \prod_{p \mid d} (1-p) \prod_{p \mid \frac{N}{d}} (1+p) = \mu(d)\sigma_{1}(N/d) \varphi(d)
		\]
		and $\mu(d) = 1$, we also obtain the stated formula for the norms.
	\end{proof}
	
	We are now ready to state a general Kronecker limit type formula for the averaged sum of hyperbolic Eisenstein series given in Theorem \ref{theorem lift 1}.

	\begin{theorem} \label{theorem hyperbolic KLF}
		Let $m>0$ and $\beta \in L'/L$ with $m \in Q(\beta)+\Z$. At $s=0$ we have the Laurent expansion
		\begin{align*}
			\sum_{X \in\Gamma_{0}(N) \backslash L_{\beta,m}}\Ehyp_{c_{X}}(z,s)
			= \sum_{d \in E(N) \setminus \{1\}}C_{\beta,m}(d) \log \Bigg( \prod_{c \mid N}|\eta(cz)|^{\mu((c,d))} \Bigg) \cdot s + O(s^{2}) ,
		\end{align*}
		where $E(N)$ is the set of all positive divisors of $N$ having an even number of prime factors, and the constants $C_{\beta,m}(d)$ are given as follows:
		\begin{enumerate}
			\item If $4Nm$ is not a square, or if $4Nm=n^2$ is a square with $N/(n,N)$ being $1$ or a prime, then the sum on the right-hand side vanishes, i.e., the sum of hyperbolic Eisenstein series has a double root at $s=0$.
			\item If $4Nm = n^2$ with $n \in \Z_{>0}$, and such that $N/(n,N)$ has at least two different prime factors, then the constants $C_{\beta,m}(d)$ are given by
			\begin{align*}
				C_{\beta,m}(d) = \begin{dcases}
					\frac{ 24n \, \mu((f,d)) }{ \sigma_{0}(N/(n,N)) \sigma_{1}(N/d) \varphi(d) } , & \text{if $(n,d)=1$} , \\ 
					0, & \text{if $(n,d)>1$} .
				 \end{dcases}
			\end{align*}
			Here we chose $f \mid N$ such that $w_{f}(n) = \beta$. In particular, the sum on the right-hand side of the above Laurent expansion does not vanish in this case.
		\end{enumerate}
	\end{theorem}
	
	\begin{proof}
		Combining Theorem \ref{theorem lift 1}, Proposition \ref{proposition theta lift continuation} and Theorem \ref{theorem Borcherds products} we obtain that the Laurent expansion of the left-hand side at $s=0$ is given by
		\[
			\sum_{X \in\Gamma_{0}(N) \backslash L_{\beta,m}}\Ehyp_{c_{X}}(z,s)
			= \frac{1}{4} \Phi(z,P_{\beta,m}) \cdot s + O(s^2)
			= - \log \abs{ \Psi(z,P_{\beta,m}) } \cdot s + O(s^2) .
		\]
		In order to compute the Borcherds product $\Psi(z,P_{\beta,m})$, we write $P_{\beta,m}$ in terms of the orthogonal basis given in Lemma \ref{lemma theta_d}, and use Lemma \ref{lemma Borcherds products theta} to obtain
		\begin{align} \label{eq Borcherds product of P for m>0}
			- \log \abs{ \Psi(z,P_{\beta,m}) } = -2 \sum_{d \in E(N) \setminus \{1\}} \frac{(P_{\beta,m},\theta_d)}{(\theta_d,\theta_d)} \log \Bigg( \prod_{c \mid N} |\eta(cz)|^{\mu((c,d))} \Bigg) .
		\end{align}
		If $4Nm$ is not a square we have $(P_{\beta,m},\theta_d)=0$ for all $d \in E(N) \setminus \{1\}$ by the coefficient formula \eqref{PeterssonCoefficientFormula} since the Fourier coefficients of the theta functions $\theta_d$ are supported on indices of the form $n^{2}/4N$ for $n \in \Z$.
		
		Let now $4Nm=n^2$ with $n \in \Z_{>0}$. Then we find $f \mid N$ such that $w_{f}(n) = \beta$ as $\beta^2 \equiv n^2$ mod $2N$. Moreover, we have $P_{\beta,m} = P_{n,n^{2}/4N}^{w_{f}}$ and $\theta_{d}^{w_{f}} = \theta_{d\star f} = \mu((f,d))\theta_{d}$. Using again formula \eqref{PeterssonCoefficientFormula} we thus find
		\begin{align*}
			(P_{\beta,m},\theta_{d})
			&= \mu((f,d))\sum_{c \mid N}\mu((c,d))\left(P_{n,n^{2}/4N},\theta^{w_{c}} \right) \\
			&= - \frac{4\pi n}{\sqrt{N}} \, \mu((f,d)) \sum_{c \mid N} \mu((c,d)) c_{\theta^{w_c}}(n,n^{2}/4N) .
		\end{align*}
		The latter Fourier coefficients are given in \eqref{eq Fourier coefficients of theta_sigma_c}, yielding
		\begin{align*}
			\sum_{c \mid N} \mu((c,d)) c_{\theta^{w_c}}(n,n^{2}/4N)
			&= 2 \sum_{c \mid (n,N)} \mu((c,d))
			= 2 \sigma_0((n,N/d)) \sum_{c \mid (n,d)} \mu((c,d)) .
		\end{align*}
		Now the remaining sum over the values of the character $c \mapsto \mu((c,d))$ of the group $D((n,d))$ is $1$ if $(n,d)=1$, and vanishes otherwise. Note that in the first case, i.e., if $(n,d)=1$, then $\sigma_0((n,N/d)) = \sigma_0((n,N))$. Therefore, the formula in item (2) follows if we take into account the norm $(\theta_{d},\theta_{d})$ computed in Lemma \ref{lemma theta_d}. Moreover, we see that the coefficients $C_{\beta,m}(d)$ can only vanish simultaneously if $(n,d)>1$ for all $d \in E(N) \setminus \{1\}$, i.e., if $N/(n,N)$ is $1$ or a prime, proving the second part of item (1).
		
		Finally, we note that since an eta quotient is uniquely determined by its exponents, one can check that the eta quotients $\prod_{c \mid N} |\eta(cz)|^{\mu((c,d))}$ with $d \in E(N) \setminus \{1\}$ appearing in \eqref{eq Borcherds product of P for m>0} are linearly independent (in the multiplicative sense). Hence the first term in the Laurent expansion of the averaged hyperbolic Eisenstein series indeed vanishes if and only if $C_{\beta,m}(d)=0$ for all $d \in E(N) \setminus \{1\}$.
	\end{proof}
	
	\begin{remark} \label{remark P=0}
		The proof of the previous theorem actually shows that for $m>0$ the Poincar\'e series $P_{\beta,m}$ vanishes exactly if $4Nm$ is not a square, i.e., if the geodesics $c_X$ with $X \in L_{\beta,m}$ are closed in the modular curve $\Gamma_0(N) \backslash \H$, or if $4Nm=n^2$ is a square with $N/(n,N)$ being $1$ or a prime. In the following we explain the vanishing of the Poincar\'e series in the latter case:
		
		If $N=1$ or $N=p$ is a prime we trivially have $P_{\beta,m}=0$ as $S_{1/2,\rho} = \{0\}$ by Lemma \ref{lemma theta_d}. We claim that if $(n,N)>1$, the Poincar\'e series $P_{\beta,m}$ of level $N$ can actually be written as a certain lift of a Poincar\'e series of lower level $N/q$ with $q=(n,N)$. In order to distinguish lattices, we write $L_{N}$ for the lattice of level $N$, and $L_{N/q}$ for the lattice of level $N/q$. Similarly, we will put superscripts at some places to emphasize the underlying lattice. We now define an operator
		\[
			V_q \colon M_{1/2,\rho^{N/q}} \to M_{1/2,\rho^{N}}
		\]
		acting on the Fourier expansion of $f = \sum_{\gamma \in L_{N/q}'/L_{N/q}} \sum_{n \geq 0} c_f(\gamma,n) e(n\tau) \e_\gamma \in M_{1/2,\rho^{N/q}}$ via
		\begin{align*}
			f | V_q = \sum_{\gamma \in L_{N}'/L_{N}} \sum_{\substack{n \in \Z+Q(\gamma) \\ n \geq 0}}
				\Bigg( \sum_{a \mid (n-Q(\gamma),\gamma,q)} c_f(\gamma/a,n/a^2) \Bigg) e(n\tau) \e_\gamma .
		\end{align*}
		Here we identify $L_N'/L_N$ with $\Z/2N\Z$, as usual. This operator is a direct translation of the corresponding $V_q$ operator acting on skew-holomorphic Jacobi forms of weight $1$ and index $N/q$ (compare \cite{EichlerZagier}, \S 4). One can check that
		\begin{align} \label{Vq lift and w_c}
			\big( f | V_q \big)^{w_c^{N}} = f^{w_c^{N/q}} | V_q
		\end{align}
		for $f \in M_{1/2,\rho^{N/q}}$ and $c \mid \frac{N}{q}$. Furthermore, comparing Fourier expansions we find
		\begin{align} \label{Vq lift and thetaN}
			\theta^{N/q} | V_q = \sum_{a \mid q} \big( \theta^{N} \big)^{w_a^{N}} ,
		\end{align}
		and combining \eqref{Vq lift and w_c} and \eqref{Vq lift and thetaN} we obtain $\theta_d^{N/q} | V_q = \theta_d^{N}$ for $d \in E(N/q)$. Conversely, computing inner products with $\theta_e^{N/q}$ for $e \in E(N/q)$ we get
		\begin{align} \label{Vq* and theta}
			\theta_d^{N} | V_q^* = \begin{cases}
				\frac{\sigma_0(q) \sigma_1(q)}{\sqrt{q}} \theta_d^{N/q} , & \text{if $(d,q)=1$} , \\
				0 , & \text{if $(d,q)>1$} .
			\end{cases}
		\end{align}
		for $d \in E(N)$. Here $V_q^*$ denotes the adjoint operator of $V_q$. Finally, we can use \eqref{Vq* and theta} and the fact that $(\beta,N) = (n,N)=q$ to prove
		\begin{align*}
			P_{\beta,n^2/4N}^{N} = \frac{q}{\sigma_1(q)} \ P_{\beta/q,(n/q)^2/(4N/q)}^{N/q} | V_q .
		\end{align*}
		Thus, if $N/q$ is $1$ or a prime, the Poincar\'e series on the right-hand side vanishes identically since the corresponding space of cusp forms is trivial, and hence the Poincar\'e series on the left-hand side vanishes as well.
	\end{remark}
	
	\begin{remark}
		 It is an interesting problem to investigate the second order coefficient of the Laurent expansion of the hyperbolic Eisenstein series at $s = 0$. In \cite{JorgensonKramerPippich2009}, Theorem 4.1, the authors proved that for a closed geodesic $c$ the Petersson inner product of $\Ehyp_{c}(z,s)$ against a Maass form $\psi$ with eigenvalue $\lambda = \frac{1}{4}+r^{2}$ is given by
		 \begin{align*}
			(\Ehyp_{c}(z,s),\psi) = \sqrt{\pi} \, \frac{\Gamma((s-1/2+ir)/2)\Gamma((s-1/2-ir)/2)}{\Gamma^{2}(s/2)}\int_{c}\psi(z)ds_{\text{hyp}}(z).
		 \end{align*}		 
		 Using similar arguments one can show an analogous formula for the hyperbolic Eisenstein series associated to an infinite geodesic $c$ if $\psi$ is a cusp form. For $\lambda \neq 0$, the right-hand side has at least a double root at $s = 0$, so the second order coefficient of the Laurent expansion of $\Ehyp_{c}(z,s)$ at $s = 0$ represents the functional
		 \[
		 \psi \mapsto \int_{c}\psi(z) ds_{\text{hyp}}(z)
		 \]
		 on Maass cusp forms with eigenvalue $\lambda$. We made some progress on its explicit construction using a theta lift in signature $(2,2)$. The details will be discussed in a subsequent work.
	\end{remark}
	
	\subsection{The elliptic case}
	
	Let $m < 0$ and $\beta \in L'/L$ with $m \in Q(\beta) + \Z$. Theorem \ref{theorem P_beta,m(tau)} tells us that the Poincar\'e series $P_{\beta,m}$ is the unique harmonic Maass form of weight $1/2$ for $\rho$ which is orthogonal to cusp forms and has principal part $e(mz)(\e_{\beta}+\e_{-\beta})$. For general squarefree $N$, the function $P_{\beta,m}$ is a proper harmonic Maass form, i.e., not weakly holomorphic, and the associated Borcherds product can only be described in terms of its roots and poles on $\H$ and at the cusps. Note that the Fourier coefficients of the holomorphic part of $P_{\beta,m}$ are real by Lemma \ref{lemma P(tau,0)}, so the corresponding Borcherds product is well-defined.
	
	\begin{theorem}\label{theorem elliptic KLF}
		Let $m<0$ and $\beta \in L'/L$ with $m \in Q(\beta)+\Z$. At $s=0$ we have the Laurent expansion
		\[
			\sum_{X \in \Gamma_{0}(N) \backslash L_{\beta,m}} \Eell_{z_{X}}(z,s) = -\log\left|\Psi_{\beta,m}(z)\right|\cdot s + O(s^{2}),
		\]
		where $\Psi_{\beta,m} \colon \H \to \C$ is a weakly holomorphic modular form of weight $0$, level $\Gamma_0(N)$ and some unitary character. Further, $\Psi_{\beta,m}$ is uniquely determined by the following properties:
		\begin{enumerate}
			\item The roots of $\Psi_{\beta,m}$ in $\H$ are located at the Heegner points $z_{X}$ for $X \in L_{\beta,m}$, with order $2$ if $\beta \equiv - \beta$ mod $L$ and order $1$ otherwise. 
			\item The order of $\Psi_{\beta,m}$ at the cusp $1/c$ for $c \mid N$ is given by
			\[
				\ord_{1/c}(\Psi_{\beta,m}) = -\frac{c}{N} \cdot \frac{H_{N}(\beta,m)}{\sigma_{0}(N)}
			\]
			where $\sigma_{0}(N) = \sum_{c \mid N}1$, and
			\[
				H_{N}(\beta,m) = \sum_{X \in \Gamma_{0}(N)\backslash L_{\beta,m}} \frac{2}{|\Gamma_0(N)_X|}
			\]
			is a Hurwitz class number.
			\item The leading coefficient in the Fourier expansion of $\Psi_{\beta,m}$ at $\infty$ is $1$.
		\end{enumerate}
	\end{theorem}
	
	\begin{proof}
		As in the proof of Theorem \ref{theorem hyperbolic KLF}, a combination of Theorem \ref{theorem lift 1}, Proposition \ref{proposition theta lift continuation} and Theorem \ref{theorem Borcherds products} yields that the Laurent expansion of the averaged sum of elliptic Eisenstein series at $s=0$ is of the form
		\[
			\sum_{X \in\Gamma_{0}(N)\backslash L_{\beta,m}}\Eell_{z_{X}}(z,s)
			= - \log \abs{\Psi_{\beta,m}} \cdot s + O(s^2)
		\]
		with $\Psi_{\beta,m}(z) = \Psi(z,P_{\beta,m})$ a meromorphic modular form of weight $c_{P_{\beta,m}}(0,0)=0$, level $\Gamma_0(N)$ and some unitary character. Its product expansion implies property (3). Moreover, as the principal part of $P_{\beta,m}$ is given by $e(m\tau) (\e_\beta+\e_{-\beta})$ we have
		\[
			\frac{1}{2} \sum_{\gamma \in L'/L} \sum_{\substack{n \in \Z+Q(\gamma) \\ n<0}} c_{P_{\beta,m}}^+(\gamma,n) \sum_{X \in L_{\gamma,n}} (z_X)
			= \frac{1}{2} \sum_{X \in L_{\beta,m}} \left[ (z_X) - (z_{-X}) \right]
			= \sum_{X \in L_{\beta,m}} (z_X) .
		\]
		Thus part (2) of Theorem \ref{theorem Borcherds products} yields property (1). In order to determine the order of $\Psi_{\beta,m}$ at the cusp $1/c$ using part (3) of Theorem \ref{theorem Borcherds products} we need to compute the Weyl vector associated to $P_{\beta,m}$ and the cusp $1/c$, i.e.,
		\[
			\rho_{P_{\beta,m},1/c} = \frac{\sqrt{N}}{8\pi} (P_{\beta,m},\theta^{w_c})^{\reg} .
		\]
		Since $\theta^{w_{c}} - \frac{1}{\sigma_{0}(N)}\sum_{d \mid N}\theta^{w_{d}}$ is a cusp form, and $P_{\beta,m}$ is orthogonal to cusp forms by Theorem \ref{theorem P_beta,m(tau)}, we can write
		\begin{align} \label{eq rho_P,1/c = sum of inner products of theta_d}
			\rho_{P_{\beta,m},1/c} = \frac{\sqrt{N}}{8\pi \sigma_{0}(N)}\bigg(P_{\beta,m},\sum_{d \mid N} \theta^{w_{d}}\bigg)^{\reg}.
		\end{align}
		In order to compute the latter inner product, we introduce Zagier's non-holomorphic Eisenstein series $E_{3/2}$ of weight $3/2$, which was first studied by Zagier for level $N = 1$ in \cite{ZagierEisensteinSeries}, and later generalized to arbitrary $N$ by Bruinier and Funke, see \cite{BruinierFunke06}, Remark 4.6 (i). It is a harmonic Maass form of weight $3/2$ for the dual representation $\bar{\rho}$ with holomorphic part
		\begin{align*}
			E_{3/2}^{+}(\tau) = \sum_{\gamma \in L'/L} \sum_{\substack{n \in \Z+Q(\gamma) \\ n \geq 0}} H_N(\gamma,-n) e(n\tau) \e_\gamma ,
		\end{align*}
		where $H_N(\gamma,-n)$ is the Hurwitz class number, which is defined in the theorem for $n<0$, and for $n=0$ we set $H_{N}(0,0) = -\sigma_1(N)/6$. Furthermore, the Eisenstein series $E_{3/2}$ is orthogonal to cusp forms with respect to the regularized inner product, and its image under the differential operator $\xi_{3/2}$ is given by
		\begin{align} \label{EisensteinXiImage}
			\xi_{3/2} E_{3/2}(\tau) = -\frac{\sqrt{N}}{4\pi} \sum_{d \mid N} \theta^{w_d}(\tau) ,
		\end{align}
		which can be checked using the Fourier expansion of $E_{3/2}$ given in \cite{BruinierFunke06}. Therefore, by equation \eqref{eq rho_P,1/c = sum of inner products of theta_d} and Stokes' theorem (applied as in \cite{BruinierFunke04}, Proposition 3.5) we find
		\begin{align*}
			\rho_{P_{\beta,m},1/c}
			&= - \frac{1}{2 \sigma_0(N)} (P_{\beta,m},\xi_{3/2}E_{3/2})^{\reg} \\
			&= \frac{1}{2 \sigma_0(N)} (E_{3/2},\xi_{1/2}P_{\beta,m})^{\reg}
				- \frac{1}{2 \sigma_0(N)} \sum_{\gamma \in L'/L} \sum_{\substack{n \in \Z+Q(\gamma) \\ n \leq 0}} c^{+}_{P_{\beta,m}}(\gamma,n) c^{+}_{E_{3/2}}(\gamma,-n) .
		\end{align*}
		The remaining inner product vanishes since $\xi_{1/2}P_{\beta,m}$ is a cusp form by Theorem \ref{theorem P_beta,m(tau)}, and as the principal part of $P_{\beta,m}$ is given by $e(m\tau) (\e_\beta+\e_{-\beta})$, we are left with
		\[
			\rho_{P_{\beta,m},1/c}
			= - \frac{1}{2 \sigma_0(N)} \left( c^{+}_{E_{3/2}}(\beta,-m) + c^{+}_{E_{3/2}}(-\beta,-m) \right)
			= -\frac{H_{N}(\beta,m)}{\sigma_{0}(N)} .
		\]
		Here we also used that $H_{N}(\beta,m) = H_{N}(-\beta,m)$. This concludes the proof.
	\end{proof}
	
	\begin{remark}
		We emphasize that the character of the Borcherds product $\Psi_{\beta,m}$ may have infinite order. Since $P_{\beta,m}$ is orthogonal to cusp forms, Theorem 6.2 and the subsequent remark from \cite{BruinierOno} show that the character has finite order if and only if the Fourier coefficients $c_{P_{\beta,m}}^+(n,n^2/4N)$ for $n \geq 1$ are all rational. However, since these coefficients are expected to be transcendental if $P_{\beta,m}$ is not weakly holomorphic, the character will typically have infinite order.
	\end{remark}
	
	\begin{remark}
		Let $N \in \Z_{>0}$ be such that the generalized Fricke group $\Gamma_0^*(N)$ has genus $0$. Then there is a normalized Hauptmodul $j_{N}^{*} = e(-z) + O(e(z))$ for $\Gamma_0^*(N)$, and the explicit formula given in Theorem \ref{theorem elliptic KLF introduction} in the introduction follows by noting that the function
		\[
			\prod_{X \in \Gamma_{0}(N) \backslash L_{\beta,m}} \left( \big(j_N^*(z)-j_N^*(z_X)\big)^{2 / |\Gamma_0(N)_X|} \right)^{1/\sigma_0(N)}
		\]
		is a weakly holomorphic modular form of weight $0$ and level $\Gamma_0(N)$, which satisfies the properties (1)--(3) of Theorem \ref{theorem elliptic KLF}, and hence agrees with $\Psi_{\beta,m}$.
	\end{remark}


\bibliography{bibliography}

\begin{thebibliography}{EMOT55}

\bibitem[AS84]{Abramowitz1984}
Milton Abramowitz and Irene~A. Stegun.
\newblock {\em {Pocketbook of Mathematical Functions}}.
\newblock Verlag Harri Deutsch, Frankfurt am Main, 1984.

\bibitem[BF04]{BruinierFunke04}
Jan~H. Bruinier and Jens Funke.
\newblock On two geometric theta lifts.
\newblock {\em Duke Math. J.}, 125(1):45--90, 2004.

\bibitem[BF06]{BruinierFunke06}
Jan~H. Bruinier and Jens Funke.
\newblock Traces of {CM} values of modular functions.
\newblock {\em J. Reine Angew. Math.}, 594:1--33, 2006.

\bibitem[BK01]{BruinierKuss2001}
Jan~H. Bruinier and Michael Kuss.
\newblock {Eisenstein series attached to lattices and modular forms on
  orthogonal groups}.
\newblock {\em Manuscripta Mathematica}, 106(4):443--459, 2001.

\bibitem[BK03]{BruinierKuehn2003}
Jan~H. Bruinier and Ulf K\"uhn.
\newblock Integrals of automorphic {G}reen's functions associated to {H}eegner
  divisors.
\newblock {\em Int. Math. Res. Not.}, 2003:1687--1729, 2003.

\bibitem[BO10]{BruinierOno}
Jan~H. Bruinier and Ken Ono.
\newblock Heegner divisors, {$L$}-functions and harmonic weak {M}aass forms.
\newblock {\em Ann. of Math. (2)}, 172(3):2135--2181, 2010.

\bibitem[Bor98]{Borcherds}
Richard~E. Borcherds.
\newblock Automorphic forms with singularities on {G}rassmannians.
\newblock {\em Invent. Math.}, 132(3):491--562, 1998.

\bibitem[Bru02]{Bruinier2002}
Jan~H. Bruinier.
\newblock {\em Borcherds products on $O(2,l)$ and Chern classes of {H}eegner
  divisors}, volume 1780 of {\em Lecture Notes in Mathematics}.
\newblock Springer-Verlag, Berlin, Heidelberg, 2002.

\bibitem[BS17]{BruinierSchwagenscheidt}
Jan~H. Bruinier and Markus Schwagenscheidt.
\newblock Algebraic formulas for the coefficients of mock theta functions and
  {W}eyl vectors of {B}orcherds products.
\newblock {\em Journal of Algebra}, 478:38--57, 2017.

\bibitem[EMOT55]{Erdelyi1955}
Arthur Erd{\'{e}}lyi, Wilhelm Magnus, Fritz Oberhettinger, and Francesco~G.
  Tricomi.
\newblock {\em {Higher Transcendental Functions. Volume 1}}.
\newblock McGraw-Hill Book Company, 1955.

\bibitem[EZ85]{EichlerZagier}
Martin Eichler and Don Zagier.
\newblock {\em The theory of {J}acobi forms}, volume~55 of {\em Progress in
  Mathematics}.
\newblock Birkh\"auser Boston Inc., Boston, MA, 1985.

\bibitem[GS83]{GoldfeldSarnak1983}
Dorian Goldfeld and Peter Sarnak.
\newblock {Sums of Kloosterman sums}.
\newblock {\em Inventiones Mathematicae}, 71(2):243--250, jun 1983.

\bibitem[JK04]{KramerJorgenson2004}
Jay Jorgenson and J{\"{u}}rg Kramer.
\newblock {Canonical metrics, hyperbolic metrics and Eisenstein series on
  $\PSL_2(\R)$}.
\newblock {\em unpublished preprint}, 2004.

\bibitem[JK11]{KramerJorgenson2011}
Jay Jorgenson and J{\"{u}}rg Kramer.
\newblock {Sup-norm bounds for automorphic forms and Eisenstein series}.
\newblock In {\em Arithmetic Geometry and Automorphic Forms}, pages 407--444,
  Boston, 2011. International Press.

\bibitem[JKvP10]{JorgensonKramerPippich2009}
Jay Jorgenson, J{\"{u}}rg Kramer, and Anna-Maria von Pippich.
\newblock {On the spectral expansion of hyperbolic Eisenstein series}.
\newblock {\em Mathematische Annalen}, 346(4):931--947, 2010.

\bibitem[JST16]{JorgensonSmajlovicThen}
Jay Jorgenson, Lejla Smajlovi\'c, and Holger Then.
\newblock Kronecker's limit formula, holomorphic modular functions, and
  q-expansions on certain arithmetic groups.
\newblock {\em Experimental Mathematics}, 25(3):295--320, 2016.

\bibitem[JvPS15]{JorgensonPippichSmajlovic2015}
Jay Jorgenson, Anna-Maria von Pippich, and Lejla Smajlovi{\'c}.
\newblock {A}pplications of {K}ronecker's limit formula for elliptic
  {E}isenstein series.
\newblock \texttt{arXiv:1505.02812 [math.NT]}, 2015.

\bibitem[JvPS16]{JorgensonPippichSmajlovic}
Jay Jorgenson, Anna-Maria von Pippich, and Lejla Smajlovi\'c.
\newblock {On the wave representation of hyperbolic, elliptic, and parabolic
  Eisenstein series}.
\newblock {\em Advances in Mathematics}, 288:887--921, 2016.

\bibitem[KM79]{KudlaMillson1979}
Stephen~S. Kudla and John~J. Millson.
\newblock Harmonic differentials and closed geodesics on a {R}iemann surface.
\newblock {\em Inventiones mathematicae}, 54:193--212, 1979.

\bibitem[Mat99]{Matthes1999}
Roland Matthes.
\newblock {On some Poincar{\'{e}}-series on hyperbolic space}.
\newblock {\em Forum Mathematicum}, 11(4):483--502, 1999.

\bibitem[{Pri}99]{DeAzevedoPribitkin1999}
Wladimir De~Azevedo {Pribitkin}.
\newblock {The Fourier coefficients of modular forms and Niebur modular
  integrals having small positive weight, I}.
\newblock {\em Acta Arithmetica}, 91(4):291--309, 1999.

\bibitem[Pri00]{Pribitkin2000}
Wladimir de~Azevedo Pribitkin.
\newblock {A generalization of the Goldfeld-Sarnak estimate on Selberg's
  Kloosterman zeta-function}.
\newblock {\em Forum Mathematicum}, 12(4):449, 2000.

\bibitem[Roe66]{Roelcke1966}
Walter Roelcke.
\newblock {Das Eigenwertproblem der automorphen Formen in der hyperbolischen
  Ebene, I}.
\newblock {\em Mathematische Annalen}, 167:292--337, 1966.

\bibitem[Roe67]{Roelcke1967}
Walter Roelcke.
\newblock {Das Eigenwertproblem der automorphen Formen in der hyperbolischen
  Ebene, II}.
\newblock {\em Mathematische Annalen}, 168:261--324, 1967.

\bibitem[Sel56]{SelbergHarmonicAnalysis}
Atle Selberg.
\newblock Harmonic analysis and discontinuous groups in weakly symmetric
  {R}iemannian spaces with applications to {D}irichlet series.
\newblock {\em J. Indian Math. Soc.}, 20:47--87, 1956.

\bibitem[Sel65]{Selberg1965}
Atle Selberg.
\newblock {On the estimation of Fourier coefficients of modular forms}.
\newblock In {\em Proceedings of Symposia in Pure Mathematics, Vol. 8}, pages
  1--15. American Mathematical Society, 1965.

\bibitem[SZ88]{SkoruppaZagier}
Nils-Peter Skoruppa and Don Zagier.
\newblock Jacobi forms and a certain space of modular forms.
\newblock {\em Inventiones Mathematicae}, 94:113--146, 1988.

\bibitem[V{\"{o}}l]{Voelz}
Fabian V{\"{o}}lz.
\newblock {\em Hyperbolic and elliptic Eisenstein series as theta lifts}.
\newblock Phd thesis, Technische Universit{\"{a}}t Darmstadt, in preparation.

\bibitem[vP10]{Pippich2010}
Anna-Maria von Pippich.
\newblock {\em The Arithmetic of Elliptic {E}isenstein Series}.
\newblock Phd thesis, Humboldt-Universit{\"{a}}t zu Berlin, 2010.

\bibitem[vP16]{Pippich2016}
Anna-Maria von Pippich.
\newblock A {K}ronecker limit type formula for elliptic {E}isenstein series.
\newblock \texttt{arXiv:1604.00811 [math.NT]}, 2016.

\bibitem[Zag75]{ZagierEisensteinSeries}
Don Zagier.
\newblock Nombres de classes et formes modulaires de poids $3/2$.
\newblock {\em C.R. Acad. Sci. Paris (A)}, 281, 1975.

\bibitem[Zag02]{ZagierTraces}
Don Zagier.
\newblock Traces of singular moduli.
\newblock In {\em Motives, polylogarithms and {H}odge theory, {P}art {I}
  ({I}rvine, {CA}, 1998)}, volume~3 of {\em Int. Press Lect. Ser.}, pages
  211--244. Int. Press, Somerville, MA, 2002.

\end{thebibliography}
\bibliographystyle{alpha}

\end{document}